\documentclass[a4paper,10pt,english]{article}

\usepackage[utf8]{inputenc}
\usepackage{babel}
\usepackage{amsfonts,amsmath,amsthm,amssymb}
\usepackage{graphicx}
\usepackage[all]{xy}
\usepackage{latexsym}
\usepackage{pstricks}
\usepackage{pstricks-add} % pour \psaxes
\usepackage{fullpage}
\usepackage{mathtools}
\usepackage{comment}
\usepackage{version}
\usepackage{pifont}
\usepackage{mathrsfs}
\title{\bf{Conditioning diffusions with respect to partial observations}}
\author{Jean-Louis Marchand\\{\small\emph{IRMAR, Université Rennes 1, Campus de Beaulieu, 35042 Rennes Cedex, France}}} 

\date{}
\newtheorem{thm}{Theorem}
\newtheorem{lem}{Lemma}
\newtheorem{prop}{Proposition}

\newcommand{\e}{\mathbb{E}}
\newcommand{\re}{\mathbb{R}}

\newcommand{\n}{\mathbb{N}}

\newcommand{\tr}{\textrm{Tr}}
\newcommand{\ud}{\mathrm{d}}

\begin{document}
\maketitle

\begin{abstract}
In this paper, we prove a result of equivalence in law between a diffusion conditioned with respect to partial observations and an auxiliary process.
By partial observations we mean coordinates (or linear transformation) of the process at a finite collection of deterministic times.
 Apart from the theoritical interest, this result allows to simulate the conditioned diffusion through Monte Carlo's method, using the fact that the auxiliary process is easy to simulate. 
\end{abstract}

\section{Introduction}

We are interessed in multidimensional diffusions solutions of stochastic differential equations (SDE's) generated by a Brownian motion.
For a $n$-dimensional diffusion solution on $[0,T]$ of the following
\begin{equation}\label{init}
\ud x_t=b_t(x_t)\ud t+\sigma_t(x_t)\ud w_t,\quad x_0=u
\end{equation} where $w$ is a $n$-dimensional Brownian  motion, it is known  (see e.g. \cite{LZ})  that its conditional law \mbox{$\mathscr{L}(x|x_T=v)$} is given by the law of a bridge process (as extension of Brownian bridge)  $y$ solution of
\begin{equation*}
\ud y_t=b_t(y_t)\ud t+\sigma_t(y_t)\ud \tilde{w}_t+\sigma_t(y_t)\sigma_t(y_t)^*\nabla_{\!\! z}\! \log p_{t,T}(z,v)\big|_{ z=y_t}\ud t,\quad y_0=u
\end{equation*}
where $\tilde{w}$ is a Brownian motion and $p_ {s,t}(z, .)$  is the density of $x_t$ knowing $x_s=z$. But in most cases this density is not explicitely known so that we are not able to simulate it easily. For practical purposes, \emph{e.g.} parameter estimation of diffusion processes, simulation of paths corresponding to the conditional law is needed.  

 In their paper \cite{DH}, B.Delyon and Y.Hu studied the following equation on  $[0,T]$
\begin{align}\label{dh}\ud y_t=b_t(y_t)\ud t-\frac{y_t-v}{T-t}\ud t+\sigma_t(y_t)\ud \tilde{w}_t,\quad y_0=u.
\end{align}where $ \tilde{w}$ is a $n$-dimensional Brownian motions.
Under adequate assumptions, the process $y$ is unique on $[0,T]$,  $\lim_{t\to T}y_t=v$, a.s. and for all positive function $f$ in $C([0,T],\re^n)$ we have
\begin{equation*}
\e[f(x)|x_T=v]=C\e[f(y)R(y)]
\end{equation*}where $R$ is a functional of whole path $y$ on $[0,T]$. The quantity $R(y)$ is computable knowing parameters $b$, $\sigma$, $T$ and $v$.
The constant $C$ is unknown, but in practice the conditional law is estimated through  
 \begin{equation*}
\e[f(x)|x_T=v]\simeq \frac{\sum_i f(y^i)R(y^i)}{\sum_i R(y^i)}
\end{equation*}where each $y^i$ is an independant sample of (\ref{dh}).
In this case, we call the process $y$  a bridge even if $y$ does not have the right targeted law. If $b= 0$ and $\sigma= I_n$ (identity $n$-dimensional matrix), the process $x$  is a $n$-dimensional Brownian motion and process $y$ is a $n$-dimensional Brownian bridge so that $C=R= 1$. 
This theorem applies in the case of more than one observation. The Markov property indeed implies that the conditional law is the tensor product of each bridge.

The aim of this paper is to extend this result to solve this problem with only partial observations. The previous remark does not apply; indeed we have to treat simultaneously all conditionings . To give an idea, let $w=(\begin{smallmatrix}w^1\\w^2\end{smallmatrix})$ be a 2-dimensional Brownian motion. The law of $w$ conditioned on $w^1_S=u$ and $w^2_T=v$ with $S<T$ is given by that one of $y$ solution of
\begin{equation*}
\ud y_t=\ud \tilde{w}_t -
\begin{pmatrix}
\frac{y_t^1-u}{S-t}\mathbf{1}_{t<S}\\\frac{y^2_t-v}{T-t}\mathbf{1}_{t<T}
\end{pmatrix}\ud t,\quad y_0=u
\end{equation*}each coordinate is a Brownian bridge.

Let us define our observations. At each deterministic positive observation time of the sequence  $0<T_1<\dots<T_k<\dots <T_N=T$, we get a partial information given by a linear transformation of $x_{T_k}$, $L_kx_{T_k}$, where  $L_k$ is a deterministic matrix in $M_{m_k,n}(\re)$  whose $m_k$ rows form an orthonormal family.
So that our aim is to be able to describe the conditional law  \mbox{$\mathscr{L}(x|(L_kx_{T_k}=v_k)_{1\leq k\leq N})$} where $v_k$ is an arbitrary  deterministic $m_k$-dimensional vector.

We define process $y$ to be the solution of
\begin{equation}\label{genpb}
\left\{
\begin{array}{ll}
\ud y_t=b_t(y_t)\ud t+\sigma_t(y_t)\ud \tilde{w}_t-\sum_{k=1}^N P_t^k(y_t) \frac{y_t-u_k}{T_k-t}\mathbf{1}_{(T_k-\varepsilon_k,T_k)}(t)\ud t\\y_0=u
\end{array}\right.
\end{equation}where for all time $t$, all vector $z$ and for all $1\leq k\leq N$, the matrix $P_t^k(z)$ is an oblique projection and $u_k$ is any vector satisfying $L_ku_k=v_k$. The correction term operates only on the interval $(T_k-\varepsilon_k,T_k)$ where $T_k-\varepsilon_k<T_k$ for technical reasons. We will show that with a good choice for those projections (see Equation (\ref{goodproj})) we have the following equivalence in law
\begin{equation*}
\begin{array}{|ccc|}\hline
&&\\
&\mathscr{L}(x|(L_kx_{T_k}=v_k)_{1\leq k\leq N})\sim \mathscr{L}(y) &\\
&&\\\hline
\end{array}
\end{equation*}with an explicit density (Theorem~\ref{th} below).

\bigskip
In this paper a first part is devoted to the study of general bridges which will provide us the good candidate whose law is absolutely continuous with respect to targeted one. The second one provides the main result. Some properties and proofs are postponed in the appendix to ease the reading. 
\paragraph{Notations} For the sake of readibility, we choose not to specify arguments when not necessary. For example (\ref{init}) becomes
\begin{equation*}
\ud x_t=b_t\ud t+\sigma_t\ud w_t
\end{equation*}
For all $z$, the matrix $a_t(z)$ is defined by
\begin{equation*}
a_t(z)
=\sigma_t(z)\sigma_t(z)^*
\end{equation*}we suppose that there exists a positive number $\rho$ such that for all $(t,z)$
\begin{equation*}
\rho^{-1}I_n<a_t(z)<\rho I_n
\end{equation*}in the sense of symmetric matrices, where $I_n$ is the $n$- dimensional identity matrix.
The function $a^{-1}$ is defined by
\begin{equation*}
\begin{array}{ccll}
a^{-1} : &[0,T]\times\re^n&\rightarrow &M_n(\re^n)\\
&(t,x) &\mapsto &\big(a_t(x)\big)^{-1}
\end{array}
\end{equation*}
We define the infimum of all the $\varepsilon_k$
\begin{equation*}
\varepsilon_0=min_k\{\varepsilon_k\}
\end{equation*}

\section{Bridges and bridges approximations} 
\subsubsection*{Bridges}

We recall that a bridge is defined as a solution of (\ref{genpb})
\begin{equation*}
\left\{
\begin{array}{ll}
\ud y_t=b_t(y_t)\ud t+\sigma_t(y_t)\ud \tilde{w}_t-\sum_{k=1}^N P_t^k(y_t) \frac{y_t-u_k}{T_k-t}\mathbf{1}_{(T_k-\varepsilon_k,T_k)}(t)\ud t\\y_0=u_0
\end{array}\right.
\end{equation*}
 We assume that the deterministic parameters $b$ and $\sigma$ are $C^{1,2}_b$ functions (bounded with bounded derivatives).
We assume that
\begin{equation*}
(t,z)\mapsto P_t^k(z)
\end{equation*}is a $C_b^{1,2}$ function and that for any $z$
\begin{equation}\label{propproj}
L_kP_t^k(z)=L_k\quad \textrm{and}\quad \ker(L_k)=\ker(P^k_t(z))
\end{equation}
First of all, a lemma to describe the behaviour of process $y$
\begin{lem}\label{cvbridge}
The SDE (\ref{genpb}) admits a unique solution on $[0,T]$ in the absolute convergence's sense meaning that
  \begin{align*}
\int_{T_k-\varepsilon_k}^{T_k}\frac{\|P_t^k(y_t)(y_t-u_k)\|}{T_k-t}\ud t<+\infty
\end{align*}
For all $k$ we have $L_ky_{T_k}=L_ku_k$ almost surely (a.s.). Moreover for $T_k-\varepsilon_k<t<T_k$, $\|L_k(y_t-u_k)\|\leq C_k(\omega)(T_k-t)\log\log [(T_k-t)^{-1}+e]$ a.s., where $C_k$ is a positive random variable.
\end{lem}

\begin{proof}
Let us remark that for times in $[T_{k-1},T_k]$ (with $T_0=0$) the SDE (\ref{genpb}) becomes
\begin{equation*}
\ud y_t=b_t\ud t-P_t^k\frac{y_t-u_k}{T_k-t}\mathbf{1}_{(T_k-\varepsilon_k,T_k)}(t)\ud t+\sigma(y_t)\ud \tilde{w}_t
\end{equation*}
So that we may reduce the proof to the study of (\ref{genpb}) with only one observation time, but we here have to consider random initial conditions. If unicity holds it will lead to the result by concatenation. The proof in the case $N=1$ is given in the appendix with Lemma~\ref{monobridconv}.
\end{proof}
\subsubsection*{Bridges approximations}
We now introduce approximations that will be useful in the proof of the main result in next section.
Let $0<\varepsilon<\varepsilon_0$, we set
\begin{align}\label{approxb}\ud y_t^\varepsilon=b_t^\varepsilon(y_t^\varepsilon)\ud t+\sigma_t(y_t^\varepsilon)\ud \tilde{w}_t-\sum_kP_t^k(y_t^\varepsilon)\frac{y^\varepsilon_t-u_k}{T_k-t}\mathbf{1}_{(T_k-\varepsilon_k,T_k-\varepsilon)}(t)\ud t,\quad y_0^\varepsilon=u_0
\end{align}
The only difference with the Bridge Equation~(\ref{genpb}) is that each correction term is stopped from a distance $\varepsilon$ from the observation time. 
\begin{lem}\label{convl2}
There exists a constant $0<\kappa<1$ such that
\begin{equation*}
\sup_{t\in[0,T]}\e[\|y_t^\varepsilon-y_t\|^2]\leq C \varepsilon^\kappa
\end{equation*}for all $0<\varepsilon<(\varepsilon_0 \wedge 1)$ where $C$ is a positive constant. The numbers $C$ and $\kappa$ depend on $T$, $N$, the $(\varepsilon_k)_k$, the $(A_k)_k$, and the bounds for $b$ and $\sigma$.
\end{lem}
\begin{proof}
Given in the appendix, the proof uses classical techniques and auxiliary processes each defined on $[T_{k-1},T_k]$.
\end{proof}
\section{Result in the case of partial observation}
\subsubsection*{Case where $b$ is bounded}
We aim to  obtain a Delyon\&Hu-type theorem that gives absolute continuity of process $x$ solution of (\ref{init}) conditioned on observations $(L_kx_{T_k}=v_k)_{1\leq k\leq N}$ with respect to a bridge process $y$ solution of (\ref{genpb}). 

We now consider a peculiar projection $P$, for all $k$ and $z$
\begin{equation}\label{goodproj}
P_t^k(z)=a_t(z)L_k^*(L_ka_t(z)L_k^*)^{-1}L_k
\end{equation}
We set
\begin{equation*}
A^k_t(z)=(L_ka_t(z)L_k^*)^{-1}
\end{equation*}
and also
\begin{equation*}
\beta_t(z)^k=\sigma_t(z)^*L_k^*A_t^k(z) \quad\textrm{and}\quad \eta_k(z)=\sqrt{\det(A_t^k(z))}
\end{equation*}Let us remark that 
\begin{equation}\label{trick}
\beta_t^k(z)^*\beta_t^k(z)=A_t^k(z)\quad \textrm{and}\quad L_k\sigma_t(z)\beta_t^k(z)=I_{m_k}
\end{equation}where $I_{m_k}$ is the identity $m_k-$dimensional matrix.
Here are both systems we now consider
\begin{align}
\ud x_t&=b_t(x_t)\ud t+\sigma_t(x_t)\ud w_t, \quad x_0=u\nonumber\\
\ud y_t&=b_t(y_t)\ud t+\sigma_t(y_t)\ud \tilde{w}_t-\sum_{k=1}^N\sigma_t(y_t)\beta_t^k(y_t)\frac{L_ky_t-v_k}{T_k-t}\mathbf{1}_{(T_k-\varepsilon_k,T_k)}\ud t, \quad y_0=u\label{bridge}
\end{align} 

The result is the following
\begin{thm}\label{th}Suppose $b$, $\sigma$ and $a^{-1}$ to be  $C_b^{1,2}$-functions. Then for any bounded continuous function $f$
\begin{multline}
\e[f(x)|(L_kx_{T_k}=v_k)_{1\leq k\leq N}]
\\=C\e\Big [ f(y)\prod_{k=1}^N \eta_k(y_{T_k})\exp \big\{-\frac{\|\beta^k_{T_k-\varepsilon_k}(L_ky_{T_k-\varepsilon_k}-v_k)\|^2}{2 \varepsilon_k}+\int_{T_k-\varepsilon_k}^{T_k}-\frac{(L_ky_s-v_k)^*L_kb_s(y_s)\ud s}{T_k-s}\\
-\frac{(L_ky_s-v_k)^*\ud\left(A_t^k(y_t)\right)(L_ky_s-v_k)}{2(T_k-s)}-\sum_{1\leq i,j\leq m_k} \frac{\ud\big\langle A^k(y_.)_{i,j},(L_ky_.-v_k)_i(L_ky_.-v_k)_j\big\rangle_s}{2(T_k-s)} \big\}\Big ]
\end{multline}where $C$ is a positive constant.

\end{thm}
\begin{proof}
This one consists in using approximations $y^\varepsilon$ solutions of (\ref{approxb}) of process $y$ solution of (\ref{bridge}). Thanks to Girsanov's theorem, we are able to obtain an equality for all bounded continuous function $f$
\begin{equation*}
\e[f(x)G^\varepsilon(x)]=\e[f(y^\varepsilon)H^\varepsilon(y^\varepsilon)]
\end{equation*}
where $G^\varepsilon/H^\varepsilon$ is the density given by Girsanov's theorem. We want to  prove that with a good choice for $G^\varepsilon$ and $H^\varepsilon$, the lefthand member of the  last inequality converges to the conditional expectation, and the righthand one converges to what  appears in the Theorem~\ref{th}.

We set for all $z\in\re^n$
\begin{equation*}
h^\varepsilon_t(z)=\sum_{k=1}^N\beta_t^k(z)\frac{v_k-L_kz}{T_k-t}\mathbf{1}_{(T_k-\varepsilon_k,T_k-\varepsilon)}(t)
\end{equation*}
Then for all bounded continuous function $f$
\begin{equation*}
\e[f(y^\varepsilon)]=\e[f(x)\exp\{-\int_0^Th^\varepsilon_t(x_t)^*\ud w_t+\frac 1 2 \|h^\varepsilon_t(x_t)\|^2\ud t\}]
\end{equation*}
We are looking for a different expression of the argument of the exponential function. We use Itô's formula for $T_k-\varepsilon_k< t< T_k-\varepsilon$ and use (\ref{trick}) to get
\begin{multline*}
\ud\left( \frac{\|\beta_t^k(x_t)(L_kx_t-v_k)\|^2}{T_k-t}\right)=\frac{2(L_kx_t-v_k)^*A_t^k(x_t)L_k\ud x_t}{T_k-t}+\frac{\|\beta_t^k(x_t)(L_kx_t-v_k)\|^2}{(T_k-t)^2}\ud t+\frac{m_k}{T_k-t}\ud t\\
+\frac{(L_kx_t-v_k)^*\ud\big(A_t^k(x_t)\big)(L_kx_t-v_k)}{T_k-t}+\sum_{1\leq i,j\leq m_k}\frac{\ud\big\langle A^k_{i,j}(x_.),(L_kx_.-v_k)_i(L_kx_.-v_k)_j\big\rangle_t}{T_k-t}
\end{multline*}
The $k^\textrm{th}$ term of $(h^\varepsilon_t)^*\ud w_t$ coming from that one in $\ud x_t$ is now isolated 
\begin{multline}\label{girsterm}
-\frac{2(L_kx_t-v_k)^*A_t^k\sigma_t\ud w_t}{T_k-t}-\frac{\|\beta_t^k(L_kx_t-v_k)\|^2}{(T_k-t)^2}\ud t= - \ud \left(\frac{\|\beta_t^k(L_kx_t-v_k)\|^2}{T_k-t}\right)+\frac{m_k}{T_k-t}\ud t\\
+\frac{2(L_kx_t-v_k)^*A_t^kb_t\ud t}{T_k-t}+\frac{(L_kx_t-v_k)^*\ud A_t^k(L_kx_t-v_k)}{T_k-t}+\sum_{1\leq i,j\leq m_k}\frac{\ud\big\langle A_{i,j}^k,(L_kx-v_k)_i(L_kx-v_k)_j\big\rangle_t}{T_k-t}
 \end{multline} 
 Since we have 
 \begin{equation*}
\|h_t^\varepsilon\|^2\ud t=\sum_k \frac{\|\beta_t^k(L_kx_t-v_k)\|^2}{(T_k-t)^2}\mathbf{1}_{(T_k-\varepsilon_k,T_k-\varepsilon)}(t)\ud t
 \end{equation*} 
 and
 \begin{equation*}
(h_t^\varepsilon)^*\ud w_t=-\sum_k\mathbf{1}_{(T_k-\varepsilon_k,T_k-\varepsilon)}(t)\frac{(L_kx_t-v_k)^*A_t^k\sigma_t\ud w_t}{T_k-t}
 \end{equation*} 
 we obtain $-2(h_t^\varepsilon)^*\ud w_t-\|h_t^\varepsilon\|^2\ud t$ adding the terms given by (\ref{girsterm}). Finally, it leads us to a new expression for the density given by Girsanov's theorem
 \begin{multline*}
\e[f(y^\varepsilon)] =\e\Big[\exp\big\{\sum_{k=1}^N- \frac{\|\beta_{T_k-\varepsilon}^k(L_kx_{T_k-\varepsilon}-v_k)\|^2}{2 \varepsilon}+\frac{\|\beta_{T_k-\varepsilon_k}^k(L_kx_{T_k-\varepsilon_k}-v_k)\|^2}{2 \varepsilon_k}\\+\int_{T_k-\varepsilon_k}^{T_k-\varepsilon}
\frac{(L_kx_t-v_k)^*A_t^kb_t\ud t}{T_k-t}+\frac{m_k}{2(T_k-t)}\ud t\\+\frac{(L_kx_t-v_k)^*\ud A_t^k(L_kx_t-v_k)}{2(T_k-t)}+\sum_{1\leq i,j\leq m_k}\frac{\ud\big\langle A_{i,j}^k,(L_kx-v_k)_i(L_kx-v_k)_j\big\rangle_t}{2(T_k-t)}\big\}\Big]
 \end{multline*} 
 In an equivalent way, even if it means changing $f$
 \begin{equation}\label{ident}
\e[f(y^\varepsilon)\varphi^\varepsilon]=\e[f(x)\psi^\varepsilon]
 \end{equation} with
 \begin{multline}\label{phie}
\varphi^\varepsilon:=\varphi^\varepsilon(y^\varepsilon)=\prod_{k=1^N}\varepsilon_k^{-\frac{m_k}{2}}\eta^\varepsilon_k(y_{T_k-\varepsilon}^\varepsilon)\exp\big\{\sum_{k=1}^N- \frac{\|\beta_{T_k-\varepsilon_k}^k(L_ky^\varepsilon_{T_k-\varepsilon_k}-v_k)\|^2}{2 \varepsilon_k}+\int_{T_k-\varepsilon_k}^{T_k-\varepsilon}-
\frac{(L_ky^\varepsilon_t-v_k)^*A_t^kb_t\ud t}{T_k-t}\\-\frac{(L_ky^\varepsilon_t-v_k)^*\ud A_t^k(L_ky^\varepsilon_t-v_k)}{2(T_k-t)}-\sum_{1\leq i,j\leq m_k}\frac{\ud\big\langle A_{i,j}^k,(L_ky^\varepsilon-v_k)_i(L_ky^\varepsilon-v_k)_j\big\rangle_t}{2(T_k-t)}\big\}
 \end{multline} and
 \begin{equation}\label{psi}
\psi^\varepsilon=C^\varepsilon\prod_{k=1}^N \eta_k^\varepsilon(x_{T_k-\varepsilon})\exp\{-\frac{\|\beta_{T_k-\varepsilon}^k(x_{T_k-\varepsilon})(L_kx_{T_k-\varepsilon}-v_k)\|^2}{2 \varepsilon}\}
 \end{equation}where
 for all $z\in\re^n$
 \begin{equation*}
\eta^\varepsilon_k(z)=\sqrt{\det\big(A_{T_k-\varepsilon}^k(z)\big)}\quad \textrm{and} \quad C^\varepsilon=\prod_k\varepsilon^{-\frac{m_k}{2}}
  \end{equation*}  
  Now using it in the case where $f=1$, we get formally
  \begin{equation*}
\frac{\e[f(y^\varepsilon)\varphi^\varepsilon]}{\e[\varphi^\varepsilon]}=\frac{\e[f(x)\psi^\varepsilon]}{\e[\psi^\varepsilon]}
  \end{equation*}  the fact that this quantity is finite is given by Proposition~\ref{psibnd}. The fact that the righthand term converges to the conditional expectation is given by Lemma~\ref{condexp} in the appendix. The proof relies essentially on the use of Aronson's estimates that provides gaussian bounds for transition probabilities.
   
   The main difficulty of the proof consists in showing  almost sure convergence and then uniform one for the $\varphi^\varepsilon$. An obvious candidate for the limit is
   \begin{multline}\label{phi}
\varphi=\prod_{k=1}^N \varepsilon_k^{-\frac{m_k}{2}}\eta_k(y_{T_k})\exp\big\{-\frac{\|\beta_{T_k-\varepsilon_k}^k(y_{T_k-\varepsilon_k})(L_ky_{T_k-\varepsilon_k}-v_k)\|^2}{2 \varepsilon_k}+\int_{T_k-\varepsilon_k}^{T_k} -\frac{(L_ky_t-v_k)^*A_t^kL_kb_t(y_t)\ud s}{T_k-t}\\-\frac{(L_ky_t-v_k)^*d\big(A_t^k(y_t)\big)(L_ky_t-v_k)}{2(T_k-t)}-\sum_{1\leq i,j\leq m_k}\frac{\ud\big\langle A_{i,j}^k(y_.),(L_ky_.-v_k)_i(L_ky_.-v_k)_j\big\rangle_t}{2(T_k-t)}\big\}
   \end{multline}   
Thanks to Lemma~\ref{convint} given in the appendix, $\varphi$ is well defined.
 As said before, we want to prove the following
 \begin{lem}\label{unifcv}There exists a decreasing sequence $(\varepsilon_i)_{i\in\n}$ tending to 0 such that
 \begin{equation*}
\lim_{i\to \infty} \e\left[\|\varphi^{\varepsilon_i}-\varphi\|\right]=0
\end{equation*}
 \end{lem} 
 \begin{proof}
The proof is decomposed into two main parts. First one aims at showing the almost sure convergence of $\varphi^{\varepsilon_i}$. In second part we prove that $\e[\varphi^{\varepsilon_i}]$ tends to $\e[\varphi]$. Finally to conclude, we will use Scheffé's lemma.

For almost sure convergence, we first use triangular inequality
\begin{equation*}
|\varphi^\varepsilon(y^\varepsilon)-\varphi(y)|\leq |\varphi^\varepsilon(y^\varepsilon)-\varphi^\varepsilon(y)|+|\varphi^\varepsilon(y)-\varphi(y)|
\end{equation*}The second one converges to 0, this is given by Lemma~\ref{convint}. We now treat the term $|\varphi^\varepsilon(y^\varepsilon)-\varphi^\varepsilon(y)|$.
\begin{multline*}
\frac{\varphi^\varepsilon(y^\varepsilon)}{\varphi^\varepsilon(y)}=\prod_{k=1
}^N \frac{\eta_k^\varepsilon(y^\varepsilon_{T_k-\varepsilon})}{\eta_k^\varepsilon(y_{T_k-\varepsilon})}\exp\Big\{-\frac{\|\beta_{T_k-\varepsilon_k}^k(y^\varepsilon_{T_k-\varepsilon_k})(L_ky^\varepsilon_{T_k-\varepsilon_k}-v_k)\|^2-\|\beta_{T_k-\varepsilon_k}^k(y_{T_k-\varepsilon_k})(L_ky_{T_k-\varepsilon_k}-v_k)\|^2}{2 \varepsilon_k}\\+\int_{T_k-\varepsilon_k}^{T_k-\varepsilon}-\frac{(L_ky_t^\varepsilon-v_k)^*A_t^k(y_t^\varepsilon)L_kb_t(y_t^\varepsilon)-
(L_ky_t-v_k)^*A_t^k(y_t)L_kb_t(y_t)}{T_k-t}\ud t\\-\frac{(L_ky_t^\varepsilon-v_k)^*\ud\big(A_t^k(y_t^\varepsilon)\big)(L_ky^\varepsilon_t-v_k)-
(L_ky_t-v_k)^*\ud\big(A_t^k(y_t)\big)(L_ky_t-v_k)}{2(T_k-t)}\\
-\sum_{i,j}\frac{\ud\big\langle A_{i,j}^k(y^\varepsilon_.),(L_ky^\varepsilon_.-v_k)_i(L_ky^\varepsilon_.-v_k)_j\big\rangle_t-\ud\big\langle A_{i,j}^k(y_.),(L_ky_.-v_k)_i(L_ky_.-v_k)_j\big\rangle_t}{2(T_k-t)}
\Big\}
\end{multline*}
We can write it respecting the order above
\begin{equation*}
\frac{\varphi^\varepsilon(y^\varepsilon)}{\varphi^\varepsilon(y)}\stackrel{Notation}{=}\prod_{k=1}^N \Xi^\varepsilon_k\exp\{\Upsilon_k^\varepsilon+\Psi_k^\varepsilon+\Theta^\varepsilon_k+\Phi_k^\varepsilon\}
\end{equation*}

According to Lemma \ref{convl2} there exists a decreasing sequence $(\varepsilon_i)_{i\in \n}$ tending to 0 satisfying for all $k$ that $y^{\varepsilon_i}_{T_k-\varepsilon_i}$ converges almost surely to $y_{T_k}$.
From this we obtain the fact that $\Xi^{\varepsilon_i}_k$ converges almost surely to 1 and $\Upsilon_k^{\varepsilon_i}$ to 0 using regularity of $\sigma$. Then for all $k$
\begin{equation*}
|\Psi_k^{\varepsilon}|\leq \int_{T_k-\varepsilon_k}^{T_k-\varepsilon}\left|\frac{L_k(y_t^{\varepsilon}-y_t)^*A_t^k(y_t^{\varepsilon})L_kb_t(y_t^{\varepsilon})}{T_k-t}\right|+\left|\frac{(L_ky_t-v_k)^*\big(A_t^k(y_t^{\varepsilon})L_kb_t(y_t^{\varepsilon})-
A_t^k(y_t)L_kb_t(y_t)\big)}{T_k-t}\right|\ud t
\end{equation*}
Since $b$ and $\sigma$ are bounded we use Lemma \ref{cvbridge} to get 
\begin{equation*}
|\Psi_k^{\varepsilon}|\leq C\left(\int_{T_k-\varepsilon_k}^{T_k-\varepsilon}\frac{\|y_t^\varepsilon-y_t\|^2}{T_k-t}\ud t\right)^{\frac{1}{2}}\left(\int_{T_k-\varepsilon_k}^{T_k-\varepsilon}\frac{(1+\log \log\big((T_k-t)^{-1}+e\big)}{T_k-t}\ud t\right)^{\frac{1}{2} }
\end{equation*}where $C$ and $C'$ are positive random variables.
Thanks to Lemma \ref{convl2}, up to  an extracted subsequence
\begin{equation*}
\lim_{i\to\infty}\int_{T_k-\varepsilon_k}^{T_k-\varepsilon_i}\frac{\|y^{\varepsilon_i}_t-y_t\|^2}{T_k-t}\ud t=0
\end{equation*}that leads us to convergence for all $k$ of $|\Psi_k^{\varepsilon_i}|$ to 0.
Now we use Identity (\ref{aexpansion})
\begin{multline*}
\Theta_k^\varepsilon=\frac{\|Ly_t-v\|^2}{T-t}p_t(y_t)\ud t+\frac{\|Ly_t-v\|^2}{T-t}q_t(y_t)\ud w_t+\frac{\|Ly_t-v\|^2}{(T-t)^2}r_t(y_t)\ud t\\-\frac{\|Ly^\varepsilon_t-v\|^2}{T-t}p_t(y^\varepsilon_t)\ud t-\frac{\|Ly^\varepsilon_t-v\|^2}{T-t}q_t(y^\varepsilon_t)\ud w_t-\frac{\|Ly^\varepsilon_t-v\|^2}{(T-t)^2}r_t(y^\varepsilon_t)\ud t
\end{multline*}where $p$, $q$, and $r$ are all $C_b^{1,2}$ functions. Hence
using Lemmas \ref{cvbridge} and \ref{convl2} as above we obtain that $\lim_{\varepsilon_i\to0}|\Theta_k^{\varepsilon_i}|=0$ up to a subsequence.
It remains to treat the term $\Phi_k^\varepsilon$. Still using Identity (\ref{aexpansion}), Lemmas \ref{cvbridge} and \ref{convl2} we show that $\lim_{\varepsilon_i\to0}|\Phi_k^{\varepsilon_i}|=0$ even if it means extracting once more a subsequence.
We have obtained almost sure convergence of $\varphi^\varepsilon$ to $\varphi$.
Then we show the convergence of the expectations. For this we set a preliminary result
\begin{prop}\label{psibnd}
There exist two positive constants $c_1$ and $c_2$ such that for all $0<\varepsilon<\varepsilon_0$
\begin{equation*}
c_1\leq C^\varepsilon\e[\psi^\varepsilon]\leq c_2
\end{equation*}
 \end{prop} 
\begin{proof}
We give an explicit expression
\begin{equation*}
C^\varepsilon\e[\psi^\varepsilon]=\int q^\varepsilon(\zeta_1,\dots,\zeta_N)\prod_{k=1}^N\eta_k^\varepsilon(\zeta_k)^{-\frac{m_k}{2}}\exp\{-\frac{\|\beta_{T_k-\varepsilon}^k(\zeta_k)(L_k\zeta_k-v_k)\|^2}{2 \varepsilon}\} \ud \zeta_k
\end{equation*}
where $q^\varepsilon$ is the density of $(x_{T_1-\varepsilon},\dots,x_{T_N-\varepsilon})$. Under theorem's assumptions $x$ is a strong Markov process, with positive transition density. For $s,t\in[0,T]$, we denote $p_{s,t}(u,z)$ the density of $x_t^{s,u}$ solution of (\ref{init}) initialized to be $u$ at time $s$. Then thanks to Aronson's estimates there exist positive constant $\mu$, $\lambda$, $M$ and  $\Lambda$ such that the density $p$ satisfies for $s<t$
\begin{equation*}
\mu(t-s)^{-\frac n 2}e^{\frac{-\lambda\|z-u\|^2}{t-s}}\leq p_{s,t}(u,z)\leq M(t-s)^{-\frac n 2}e^{\frac{-\Lambda\|z-u\|^2}{t-s}}
 \end{equation*} 
 Now using $p$ we are able to write 
 \begin{equation*}
q^\varepsilon(\zeta_1,\dots,\zeta_N)=p_{0,T_1-\varepsilon}(u,\zeta_1)\dots p_{T_{N-1}-\varepsilon,T_N-\varepsilon}(\zeta_{N-1},\zeta_N)
 \end{equation*} Then we apply Aronson's estimates and the fact that for all $i,j$ the coordinate $A^k_{i,j}$ is bounded by two positive constants. We obtain bounds for $C^\varepsilon\e[\psi^\varepsilon]$ of the type
 \begin{equation*}
\lambda^{-1} C^\varepsilon\int  \exp \{\sum_{j=1}^N \frac{-\lambda \|L_k\zeta_k-v_k\|^2}{2 \varepsilon}-\frac{\lambda \|\zeta_1-u\|^2}{T_1-\varepsilon}-\sum_{k=2}^N\frac{\lambda \|\zeta_k-\zeta_{k-1}\|^2}{T_k-T_{k-1}}\}\prod_{k=1}^N \ud \zeta_k
 \end{equation*} 
 where $\lambda$ is a positive constant large for the lower bound and small for the upper one. The integral can  be interpreted as a gaussian expectation where 
 \begin{equation*}
\begin{pmatrix}
\zeta_1\\\zeta_2-\zeta_1\\\vdots\\\zeta_N-\zeta_{N-1}
\end{pmatrix}
 \end{equation*} is a centered gaussian vector with covariance matrix
 \begin{equation*}
R^\varepsilon=\frac{1}{2\lambda}
\begin{pmatrix}
(T_1-\varepsilon)I_n &0 &\dots &0\\
0  &(T_2-T_1)I_n &\ddots &\vdots\\
\vdots & \ddots &\ddots &0\\
0 &\dots &0 &(T_N-T_{N-1})I_n
\end{pmatrix}
 \end{equation*} where $I_n$ is the $n$-dimensional identity matrix. As a remark, in the sense of symmetric matrices, there exist two positive constants $c_1$ and $c_2$ such that
 \begin{equation*}
c_1 I_{Nn}<R^\varepsilon<c_2I_{Nn}
 \end{equation*} 
 where $I_{Nn}$ is the $Nn$-dimensional identity matrix. Thus the gaussian vector
 \begin{equation*}
\begin{pmatrix}
\zeta_1\\\vdots\\\zeta_N
\end{pmatrix}
 \end{equation*} admits for covariance matrix
 \begin{equation*}
\Gamma^\varepsilon=G^{-1}R^\varepsilon G^{-*}
 \end{equation*} where
 \begin{equation*}
G=\begin{pmatrix}
I_n &0 &\dots &\dots &0\\
-I_n &\ddots &\ddots & &\vdots\\
0 &\ddots &\ddots &\ddots &\vdots\\
\vdots &\ddots &\ddots &\ddots &0\\
0 &\dots &0 &-I_n &I_n
\end{pmatrix}
 \end{equation*} 
 We still keep bounds for the covariance matrix 
 \begin{equation*}
c_1I_{Nn}<\Gamma^\varepsilon<c_2 I_{Nn}
 \end{equation*} Now we can get bounds for $C^\varepsilon\e[\psi^\varepsilon]$ with expectations of type
 \begin{equation*}
\lambda^{\frac{m_k} 2} C^\varepsilon\e[\exp\{-\sum_{k=1}^N \frac{\lambda \| L_kX_k-v_k\|^2}{2 \varepsilon}\}]
 \end{equation*} where $X_k$ is a $n$-dimensional gaussian variable. Then we use Lemma \ref{randvar} given in the appendix to obtain the fact that
 \begin{equation*}\begin{array}{cll}
\re^{m_1}\times\dots\times\re^{m_N} &\rightarrow &\re\\
(v_1,\dots,v_N) &\mapsto &C^\varepsilon\prod_{k=1}^N\lambda^{\frac{m_k}2} \e[\exp\{-\frac{\lambda \| L_kX_k-v_k\|^2}{2 \varepsilon}\}]\end{array}
 \end{equation*} 
 is a gaussian density of a variable $(L_1X_1+\sqrt{\frac \varepsilon \lambda}Y_1,\dots,L_NX_N+\sqrt{\frac \varepsilon \lambda}Y_N)$ where the $Y_k$ are $m_k$-dimensional centered normalized gaussian vectors. Moreover the two families $(X_k)_k$ and $(Y_k)_k$ are independant. Finally using bounds obtained above for $\Gamma^\varepsilon$  we get the fact that for all $0<\varepsilon<\varepsilon_0$
 \begin{equation*}
c_1< C^\varepsilon\e[\psi^\varepsilon]<c_2
 \end{equation*} 
\end{proof}
As a first consequence, thanks to identity (\ref{ident}), $\e[\varphi^\varepsilon]$ is finite  so that $\frac{\varphi^\varepsilon}{\e[\varphi^\varepsilon]}$ is a density. We may also use Fatou's lemma to get
\begin{equation*}
\e[\varphi] \leq \liminf_{\varepsilon\to 0}\e[\varphi^\varepsilon]\leq c_2 
\end{equation*}
It takes more work to control $\limsup_{\varepsilon\to 0}\e[\varphi^\varepsilon]$.
 
 Let $J>0$ be a large number, we introduce for all process $(z_t)_{t\in[0,T]}$ and for all $1\leq k\leq N$ the stopping time $\tau_k^\varepsilon$
 \begin{align*}
\tau_k^\varepsilon&=\inf \{t_k<t\leq T_k-\varepsilon\,:\,\frac{1}{\sqrt{T_k-t}}\exp\{-\frac{\|L_kz_t-v_k\|^2}{2(T_k-t)}D\}\leq J^{-1}\}\\
&=\inf\{t_k<t\leq T_k-\varepsilon\,:\,\frac{\|L_kz_t-v_k\|^2}{2(T_k-t)}\geq D^{-1}\log\left(\frac{J}{\sqrt{T_k-t}}\right)\}
 \end{align*} 
 where $D$ is a positive constant such that $DI_d\leq A_k$. We know that such a constant exists according to assumptions on the function $a$. The $t_k$ are chosen to be real numbers contained in $(T_{k-1},T_k-\varepsilon)$. As a convention we set $\tau_k^\varepsilon=T_k$ if the condition is empty. Let $\tau^\varepsilon$ be the first of the $\tau_k^\varepsilon$ such that the condition is non-empty
 \begin{equation*}
\tau^\varepsilon=\inf_k\{\tau_k^\varepsilon\,:\,\tau_k^\varepsilon<T_k\}
 \end{equation*} we set as convention $\tau^\varepsilon=T$ if if for all $k$, $\tau^\varepsilon_k=T_k$. Even if it means changing $f$ in Equation (\ref{ident})
 \begin{equation*}
\frac{\e[f(y^\varepsilon)\mathbf{1}_{\tau^\varepsilon<T}\varphi^\varepsilon]}{\e[\varphi^\varepsilon]}=\frac{\e[f(x)\mathbf{1}_{ \tau^\varepsilon<T}C^\varepsilon\psi^\varepsilon]}{\e[C^\varepsilon\psi^\varepsilon]}
 \end{equation*} We recall that
 \begin{equation*}
C^\varepsilon\psi^\varepsilon=\prod_k \varepsilon^{- \frac{m_k}{2}}\exp\{-\frac{\|\beta_{T_k-\varepsilon}^k(L_kx_{T_k-\varepsilon}-v_k\|^2}{2 \varepsilon}\}
 \end{equation*}
 We now consider to be on set $\{ \tau^\varepsilon=\tau_k^\varepsilon\}$\\
 \begin{center}\setlength{\unitlength}{1cm}
\begin{picture}(4.3,0)(-2.5,2)
\put(-5,2.4){\line(1,0){9}}
\put(-2,1.9){$t_{k}$}
\put(0,2.6){$\tau^\varepsilon$}
\put(3,1.9){$T_{k}$}
\put(-2,2.2){\line(0,1){.3}}
\put(3,2.2){\line(0,1){.4}}
\put(-4,2.2){\line(0,1){0.4}}
\put(2, 2.2){\line(0, 1){.4}}
\put(0,2.2){\line(0,1){.3}}
\put(-4.5,1.9){$T_{k-1}$}
\put(1.5,1.9){$T_k-\varepsilon$}
\multiput(-7,2.4)(0.2,0){10}{\line(1,0){.1}}
\multiput(4,2.4)(0.2,0){10}{\line(1,0){.1}}
\end{picture}
\end{center} 
We decompose $C^\varepsilon\psi^\varepsilon$ into three parts as a product of three factors
\begin{align*}
F_1&=\prod_{j<k} \varepsilon^{-\frac {m_j} 2}\exp\{-\frac{\|\beta_{T_j-\varepsilon}^j(L_jx_{T_j-\varepsilon}-v_j)\|^2}{2 \varepsilon}\}\\
F_2&=\varepsilon^{-\frac {m_k} 2}\exp\{-\frac{\|\beta_{T_k-\varepsilon}^k(L_kx_{T_k-\varepsilon}-v_k)\|^2}{2 \varepsilon}\}\\
F_3&=\prod_{j>k} \varepsilon^{-\frac {m_j} 2}\exp\{-\frac{\|\beta_{T_j-\varepsilon}^j(L_jx_{T_j-\varepsilon}-v_j)\|^2}{2 \varepsilon}\}
\end{align*}
We are interessed in 
\begin{equation*}
\e[C^\varepsilon\psi^\varepsilon\mathbf{1}_{\tau^\varepsilon=\tau^\varepsilon_k}]=\e[F_1F_2F_3\mathbf{1}_{\tau^\varepsilon=\tau^\varepsilon_k}]
\end{equation*}
We now use Markov's property to get independance between Past and Future knowing Present (cf \cite{DM} see last chapter about conditional expectations)
\begin{align}\label{fact}
\e[C^\varepsilon\psi^\varepsilon\mathbf{1}_{\tau^\varepsilon=\tau^\varepsilon_k}]&=\e[F_1\e[F_2F_3\mathbf{1}_{\tau^\varepsilon=\tau^\varepsilon_k}|\tau_k^\varepsilon,x_{\tau_k^\varepsilon}]]=\e[F_1\e[F_2\mathbf{1}_{\tau^\varepsilon=\tau^\varepsilon_k}|\tau_k^\varepsilon,x_{\tau_k^\varepsilon}]\e[F_3|x_{\tau_k^\varepsilon}]]\\&=\e[F_1\e[F_2\mathbf{1}_{\tau^\varepsilon=\tau^\varepsilon_k}|\tau_k^\varepsilon,x_{\tau_k^\varepsilon}]\e[F_3|x_{T_k-\varepsilon}]]
\end{align}
In order to study the factor $\e[F_2\mathbf{1}_{\tau^\varepsilon=\tau^\varepsilon_k}|\mathscr{F}_{\tau_k^\varepsilon}]$ we introduce
\begin{equation*}
\theta_t= \frac{1}{\sqrt{T_k-t}}\exp\{-\frac{\|\beta_t^{k}(L_kx_t-v_k)\|^2}{2(T_k-t)}\}
\end{equation*}
For $t\in [T_{k-1},T_k-\varepsilon]$, we set $z_t=L_kx_t-v_k$, $p_t=\|\beta_t^k(L_kx_t-v_k)\|$. We recall that $\beta_t^k=\sigma_t^*L_k^*A^k_t$ with
 $A_t^k=(\beta_t^k)^*\beta_t^k=(L_ka_t L_k^*)^{-1}$.
With respect to these notations, we have
\begin{equation*}
p_t^2=z_t^*A_t^k z_t
\end{equation*}It is also easy to see that
\begin{equation*}
\ud z_t=L_kb_t\ud t+L_k\sigma_t\ud w_t
\end{equation*}and then $\ud\langle z\rangle_t=L_ka_t L_k^*\ud t=(A_t^k)^{-1}\ud  t$. We use Itô's formula
\begin{align*}
\ud (p_t^2)= \ud (z_t^* A_t^k z_t)= 2z_t A^k_t\ud z_t + z_t^* \ud A^k_t z_t+\sum_{i,j}\ud \langle A^k_{i,j},z_i z_j\rangle_t + m_k\ud t
\end{align*}
Then 
\begin{equation*}
\ud \frac{p_t^2}{T_{k}-t}=\frac{2z_tA^k_t\ud z_t}{T_{k}-t}+\frac{p_t^2\ud t}{(T_{k}-t)^2}+\frac{z_t^*\ud A^k_tz_t}{T_k-t}+\frac{m_k\ud t}{T_k-t}+\frac{\sum_{i,j}\ud \langle \Delta_{i,j},z_iz_j\rangle_t}{T_k-t}
\end{equation*}
First using definitions of $z$, $\beta^k$ and $A^k$ we get
\begin{align*}
z_t^*A_t^k\ud z_t&=z_t^*A_t^kL_k\ud x_t=z_s^*A_t^kL_k\sigma_t\sigma_t^{-1}\ud x_t\\
&=z_t^*(\beta_t^k)^*\sigma_t^{-1}b_t\ud t+z_t^*(\beta_t^k)^*\ud w_t
\end{align*}
This leads us to the existence of two bounded adapted processes $r^{(1)}$ and $r^{(2)}$ defined on $[T_{k-1},T_k-\varepsilon]$ such that
\begin{equation*}
z_t^*A_t^k\ud z_t=p_t r^{(1)}_t\ud t+p_t  r^{(2)}_t\ud w_t
\end{equation*}In a same way we remark that there exist two bounded adapted processes $r^{(3)}$ and $r^{(4)}$ such that
\begin{equation*}
\ud (L_ka_tL_k^*)=r^{(3)}_t\ud t+r^{(4)}_t\ud w_t
\end{equation*}we get even if it means changing $r^{(3)}$ and $r^{(4)}$
\begin{align*}
z_t^*\ud A_t^kz_t&=z_t^*\big(\ud (L_ka_tL_k^*)^{-1}\big)z_t=p_t^2r^{(3)}_t\ud t+ p_t^2r^{(4)}_t\ud w_t
\end{align*}
Finally, we obtain existence of two bounded adapted processes $r$ and $r'$ such that
\begin{equation*}
\ud \frac{p_t^2}{T-t}=\frac{2p_t\ud w_t}{T-t}+\frac{p_t^2\ud t}{(T-t)^2}+r_t\frac{p^2_t}{T-t}\ud w_t+\frac{\ud t}{T-t}+r'_t\frac{p_t^2+p_t}{T-t}\ud t
\end{equation*}From this we deduce that quadratic variation
\begin{equation*}
\ud \left\langle\frac{p_t^2}{T_k-t}\right\rangle=\frac{4p_t^2+r_t^2p_t^4+4r_tp_t^3}{(T_k-t)^2}\ud t
\end{equation*}
Now we apply Itô's formula to  the function $\theta$ always for $t\in[T_{k-1},T_k-\varepsilon]$
\begin{equation*}
\ud \theta_t=\frac{\theta_t\ud t}{2(T_k-t)}-\frac{1}{2}\theta_t\ud \left(\frac{p_t^2}{T_k-t}\right)+\frac{1}{8}\theta_t\ud \left\langle\frac{p_t^2}{T_k-t}\right\rangle
\end{equation*}
We deduce from the three last equations after simplification of four terms that there exists a martingale $M$ and a bounded adapted process $r''$ both defined on $[T_{k-1},T_k-\varepsilon]$ such that
\begin{equation*}
\ud \theta_t=\ud M_t+\theta_tr''_t\left(\frac{p_t^2+p_t}{T_k-t}+\frac{p_t^4+p_t^3}{(T_k-t)^2}\right)\ud t
\end{equation*}
For any $\eta>0$, functions $x\mapsto e^{-\eta\frac{z^2}{2}}|z|^m$ for $m=1,2,3,4$ are all bounded, then there exists a constant $c_\eta$ such that
\begin{equation*}
(\sqrt{T_k-t}\,\theta_t)^\eta\left(\frac{p_t^2+p_t}{T_k-t}+\frac{p_t^4+p_t^3}{(T_k-t)^2}\right)\leq\frac{c_\eta}{\sqrt{T_k-t}}
 \end{equation*} This gives us the existence of a bounded adapted process $\pi$ defined on $[T_{k-1},T_k-\varepsilon]$ that allows us to write
 \begin{equation*}
\ud \theta_t=\ud M_t+\theta_t^{1-\eta}(T_k-t)^{-h}\alpha_t\ud t
 \end{equation*} with $h=\frac{1+\eta}{2}$.
 We now integrate it for $t\in (\tau_k^\varepsilon, T_k-\varepsilon]$
\begin{equation*}
\theta_t=\theta_{\tau_k^\varepsilon}+M_t-M_{\tau_k^\varepsilon}+\int_{\tau_k^\varepsilon}^t\pi_s\theta_s^{1-\eta}(T_{k}-s)^{-h}\ud s
 \end{equation*} 
This leads to the following
\begin{equation*}
\e[\theta_t\mathbf{1}_{\tau_k^\varepsilon <t}]\leq J^{-1}+\bar{\pi}\int_{t_k}^t\e[\theta_s\mathbf{1}_{\tau_k^\varepsilon <s}]^{1-\eta}(T_{k}-s)^{-h}\ud s
 \end{equation*} where $\bar{\pi}=\sup_s|\pi_s|$. So  $\e[\theta\mathbf{1}_{\tau_k^\varepsilon<t}]$ is bounded by the solution $u$ of
 \begin{equation*}
\ud u_s=\bar{\alpha}u_s^{1-\eta}(T_{k}-s)^{-h}\ud s,\quad u_{t_k}=J^{-1}
 \end{equation*} 
 and this equation has an explicit solution
 \begin{equation*}
u_t=\left\{\frac{\eta\bar{\alpha}}{1-h}[(T_{k}-t_k)^{1-h}-(T_{k}-t)^{1-h}]+J^{-\eta}\right\}^{\frac{1}{\eta}}\leq\left\{c_k(T_{k}-t_k)^{1-h}+J^{-\eta}\right\}^{\frac{1}{\eta}}
 \end{equation*} where $c_k$ is a positive constant. Then for all $t\in [T_{k-1}-T_k-\varepsilon]$
 \begin{equation*}
\e[\theta_t\mathbf{1}_{t>\tau_k^\varepsilon}|\tau_k^\varepsilon,x_{\tau_k^\varepsilon}]\leq \{c_k(T_k-t_k)^{1-h}+J^{-\eta}\}^{\eta}
 \end{equation*} In particular when $t=T_k-\varepsilon$
 \begin{equation*}
\e[F_2\mathbf{1}_{t>\tau_k^\varepsilon}|\tau_k^\varepsilon,x_{\tau_k^\varepsilon}]=\e[\theta_{T_k-\varepsilon}\mathbf{1}_{t>\tau_k^\varepsilon}|\tau_k^\varepsilon,x_{\tau_k^\varepsilon}]\leq \{c_k(T_k-t_k)^{1-h}+J^{-\eta}\}^{\eta}
 \end{equation*} 
 We now come back to equation (\ref{fact}), we get a first bound
 \begin{equation*}
\e[C^\varepsilon\psi^\varepsilon\mathbf{1}_{\tau^\varepsilon=\tau_k^\varepsilon}]\leq \{c_k(T_k-t_k)^{1-h}+J^{-\eta}\}^{\eta}\e[F_1\e[F_3|x_{T_k-\varepsilon}]]
  \end{equation*}
 In order to treat the factor $\e[F_3|x_{T_k-\varepsilon}]$ we use Aronson's estimate to get
  \begin{equation*}
\e[F_3|\mathscr{F}_{T_k-\varepsilon}]\leq G \int  \prod_{j>k} \frac{1}{\sqrt{\varepsilon}}\exp\{-\frac{ \|\beta^j_{T_j-\varepsilon}(L_j\zeta_j-v_j)\|^2}{2\varepsilon}\} \frac{1}{T_j-T_{j-1}}\exp\{-\Lambda_j\frac{|\zeta_j-\zeta_{j-1}|^2}{T_j-T_{j-1}}\}\ud \zeta_j 
  \end{equation*}where $G$ is a positive constant.  We just have to use Lemma \ref{randvar} given in the appendix to obtain an positive constant upper bound. The same Lemma \ref{randvar} brings us a positive constant upper bound for $\e[F_1]$.
  Finally the inequation we get from equation (\ref{fact}) is the following
  \begin{equation*}
\e[C^\varepsilon\psi^\varepsilon\mathbf{1}_{\tau^\varepsilon=\tau_k^\varepsilon}]\leq G\{c_k(T_k-t_k)^{1-h}+J^{-\eta}\}^{\eta}
  \end{equation*}  
 where $G$ is a positive constant. From this we deduce
 \begin{equation*}
\e[C^\varepsilon\psi^\varepsilon\mathbf{1}_{T>\tau^\varepsilon}]=\sum_k\e[C^\varepsilon\psi^\varepsilon\mathbf{1}_{\tau^\varepsilon=\tau_k^\varepsilon}]\leq G \max_k\left\{(c_k(T_k-t_k)^{1-h}+J^{-\eta})^{\eta}\right\}
 \end{equation*} 
 According to this last result and using the lower bound of $C^\varepsilon\e[\psi^\varepsilon]$ given by Proposition \ref{psibnd} we finally have
 \begin{equation*}
\frac{\e[C^\varepsilon\psi^\varepsilon\mathbf{1}_{T=\tau^\varepsilon}]}{\e[C^\varepsilon\psi^\varepsilon]}\geq 1-G \max_k\left\{(c_k(T_k-t_k)^{1-h}+J^{-\eta})^{\eta}\right\}
 \end{equation*} where $G$ is positive constant. So using inequality (\ref{ident}) we obtain
 \begin{equation}
\label{mino}\frac{\e[\varphi^\varepsilon\mathbf{1}_{T=\tau^\varepsilon}]}{\e[\varphi^\varepsilon]}\geq 1-G \max_k\left\{(c_k(T_k-t_k)^{1-h}+J^{-\eta})^{\eta}\right\}
 \end{equation}
 Moreover the family $(\varphi^\varepsilon\mathbf{1}_{T=\tau^\varepsilon})_{\varepsilon}$ is uniformly integrable. Indeed by definition of $\tau^\varepsilon$ we can get upper bounds depending on $J$ for the different factors in Expression (\ref{phie}) of $\varphi^\varepsilon$ or (\ref{phi}) of $\varphi$, for all $0\leq\varepsilon<1$
 \begin{multline*}
\varphi^\varepsilon=\prod_{k=1}^N \varepsilon
_k^{-\frac{m_k}{2}} \eta^\varepsilon_k(y^\varepsilon_{T_k-\varepsilon})\exp\{-\int_{T_k-\varepsilon_k}^{T_k-\varepsilon}
\frac{(L_ky^\varepsilon_t-v_k)^*A^k_tL_kb_t\ud t}{T_k-t}-\frac{(L_ky^\varepsilon_t-v_k)^*\ud A_t^k (L_ky^\varepsilon_t-v_k)}{2(T_k-t)}\\-\sum_{1\leq i,j\leq m_k}\frac{\ud \langle A_{i,j}^k,(L_ky^\varepsilon-v_k)_i(L_ky^\varepsilon-v_k)_j\rangle_t}{2(T_k-t)}-\frac{\|\beta_{T_k-\varepsilon_k}^k(y^\varepsilon_{T_k-\varepsilon_k})(L_ky^\varepsilon_{T_k-\varepsilon_k}-v_k)\|^2}{2 \varepsilon_k}\}
 \end{multline*} in fact $"\varphi^0=\varphi"$.
We recall that $b$ and $\sigma$ are bounded so is $\eta$. Then on $\{\tau^\varepsilon=T\}$
 \begin{equation*}
\left\|\frac{(L_ky^\varepsilon_t-v_k)^*A^k_tL_kb_t}{T_k-t}\right\|\leq \frac{ C }{\sqrt{T_k-t}}\sqrt{\log\left(\frac{J}{(T_k-t)^{\frac{m_k} 2}}\right)}
 \end{equation*} which is an integrable quantity in $T_k$, and $C$ is a positive constant depending on the choice of $b$ and $\sigma$. A same method gives an upper bound for the terms where quadratic variation appears.
 For the terms in $\ud A_t^k$, we decompose with respect to integrals with respect to $\ud t$ and $\ud w_t$
  \begin{equation*}
\frac{(L_ky^\varepsilon_t-v_k)^*\ud A_t^k (L_ky^\varepsilon_t-v_k)}{2(T_k-t)}=\frac{\|L_k^*y^\varepsilon_t-v_k\|^2}{T_k-t}r^k_t(y^\varepsilon_t)\ud w_t+\frac{\|L_ky^\varepsilon_t-v_k\|^2}{T_k-t}q_t^k(y^\varepsilon_t)\ud t
 \end{equation*} where $r^k$ and $q^k$ are bounded adapted functions.
 Then for fixed $J$, there exists a constant $K$ such that 
 \begin{equation*}\varphi^\varepsilon\mathbf{1}_{T=\tau^\varepsilon}\leq C\prod_k\exp\{\int_{T_k-\varepsilon_k}^{T_k} \frac{\|L_ky_t-v_k\|^2}{T_k-t}r^k_t\ud w_t-\frac 1 2\frac{\|L_ky_t-v_k\|^4}{(T_k-t)^2}\|r^k_t\|^2\ud t\}
\end{equation*}where $C$ is  a positive constant. Let us recall the following lemma (cf \cite{KS} p.198)
\begin{lem}[Novikov]
Let $(M_t)_{t\in\re}$ be a continuous local martingale, we set for all $t$
\begin{equation*}
Z_t=\exp\{M_t-\frac{1}{2}\langle M\rangle_t \}
\end{equation*}If
\begin{equation*}
\e[\exp\{\frac{1}{2}\langle M\rangle_t\}]<+\infty
\end{equation*}then we have 
\begin{equation*}
\e[Z_t]=1
\end{equation*}
\end{lem}Let us remark that for all
$p>0$
\begin{equation*}
\exp\{ p\sum_k\int_{T_{k}-\varepsilon_k}^{T_k-\varepsilon}\frac{\|L_ky^\varepsilon_t-v_k\|^4}{2|T_k-t|^2}\|r^k_t\|^2\ud t\}]\leq \exp\{p\sum_k\int_{T_{k}-\varepsilon_k}^{T_k-\varepsilon} \left[\log\left(\frac{J}{(T_k-t)^{\frac {m_k} 2}}\right)\right]^2\ud t\}
\leq C^p
\end{equation*}
where $C$ is a positive constant. Thus, we apply Novikov's lemma to get uniform integrability.
Then we take the $\liminf_{\varepsilon\to0}$ and use Lebesgue's theorem to obtain 
\begin{equation}
\frac{\e[\varphi\mathbf{1}_{T=\tau^\varepsilon}]}{\limsup_{\varepsilon\to 0}\e[\varphi^\varepsilon]}\geq 1-N\max_k \left\{(c_k(T_{k}-\varepsilon-t_k)^{1-h}+J^{-\eta})\right\}^{\frac{1}{\eta}}
\end{equation}
Now $\mathbf{1}_{T=\tau^\varepsilon}$ converges almost surely to 1 as $J$ tends to infinity. We are able to say after making the $t_k$ tend to $T_k$ that
\begin{equation}
\limsup_{\varepsilon\to 0}\e[\varphi^\varepsilon]\leq\e[\varphi]
\end{equation}
We finish the proof by Scheffé's Lemma (cf \cite{DM} p.36)
\begin{lem}[Scheffé]
Let $(f_n)_{n\in\n}$ be a sequence of positive functions converging to $f$, moreover we suppose that 
\begin{displaymath}
\lim_{n\to\infty}\e[f_n]=\e[f]<\infty
 \end{displaymath} then the sequence $(f_n)$ converges to $f$ in $\mathbb{L}^1$.
\end{lem}
 \end{proof}
 
 Finally we conclude thanks to Lemmas \ref{condexp} and \ref{unifcv}.
  \end{proof}

\subsection*{Case where $b$ is unbounded}

Suppose now that $b$ is locally Lipschitz with respect to $x$ and is locally bounded. Moreover the SDE (\ref{init}) admits a strong solution. We use a Girsanov theorem to reduce the problem to the case of a bounded drift.

We recall the Girsanov theorem for unbounded drifts introduced in \cite{DH}
\begin{thm}\label{girsdh}
Let $b$, $h$ and $\sigma$ be measurable functions from $\re^+\times \re^n$ to $\re^n$, $\re^d$ and $\re^{n\times d}$ locally Lipschitz with respect to $x$; consider the following SDE's
\begin{align*}
&\ud x_t=b_t(x_t)\ud t+\sigma_t(x_t)\ud w_t,\\
&\ud y_t=(b_t(y_t)+\sigma_t(y_t)h_t(y_t))\ud t+\sigma_t(y_t)\ud \tilde{w}_t,\\
&x_0=y_0
 \end{align*} 
on the finite interval $[0,T]$. We assume the existence of strong solution for each equation. We assume in addition that $h$ is bounded on compact sets. Then  the Girsanov formula holds: for any non negative Borel function $f$ defined on $C([0,T],\re^n)$, one has
\begin{align*}
&\e[f(y,\tilde{w}^h)]=\e[f(x,w)\exp\{\int_0^T h^*_t(x_t)\ud w_t-\frac 1 2 \int_0^T \|h_t(x_t)\|^2\ud t\}]\\
&\e[f(x,w)]=\e[f(y,\tilde{w}^h)\exp\{-\int_0^T h^*_t(y_t)\ud \tilde{w}_t-\frac 1 2 \int_0^T \|h_t(y_t)\|^2\ud t\}]
\end{align*}where $\tilde{w}^h=\tilde{w}_t+\int_0^th_s(y_s)\ud s$.
\end{thm}

Thanks to both Theorems \ref{th} and \ref{girsdh} we obtain
\begin{thm}
Suppose $\sigma$ and $a^{-1}$ to be  $C_b^{1,2}$-functions. Assume that $b$ is a locally Lipschitz with respect to $x$ and locally bounded function.
Let $y$ be the solution of
\begin{equation*}
\ud y_t=\hat{b}_t(y_t)\ud t+\sigma_t(y_t)\ud \tilde{w}_t -\sum_{k=1}^N \sigma_t(y_t)\beta_t^k(y_t)\frac {L_ky_t-v_k}{T_k-t}\mathbf{1}_{(T_k-\varepsilon_k,T_k)}(t)\ud t
\end{equation*}
where $\hat{b}$ satisfies the assumptions of Theorem \ref{th}.

 Then for any bounded continuous function $f$
\begin{multline*}
\e[f(x)|(L_kx_{T_k}=v_k)_{1\leq k\leq N}]
\\=C\e\Big [ f(y)\prod_{k=1}^N \eta_k(y_{T_k})\exp \big\{-\frac{\|\beta^k_{T_k-\varepsilon_k}(L_ky_{T_k-\varepsilon_k}-v_k)\|^2}{2 \varepsilon_k}+\int_{T_k-\varepsilon_k}^{T_k}-\frac{(L_ky_s-v_k)^*L_k\hat{b}_s(y_s)\ud s}{T_k-s}\\
-\frac{(L_ky_s-v_k)^*\ud\left(A_t^k(y_t)\right)(L_ky_s-v_k)}{2(T_k-s)}-\sum_{1\leq i,j\leq m_k} \frac{\ud\big\langle A^k(y_.)_{i,j},(L_ky_.-v_k)_i(L_ky_.-v_k)_j\big\rangle_s}{2(T_k-s)}\\
+\int_0^T \check{b}^*_t(y_t)a_t(y_t)^{-1}\ud y_t-\frac 1 2  \|\sigma_t(y_t)^{-1}\check{b}_t(y_t) \|^2\ud t\big\}\Big ]
\end{multline*}where $C$ is a positive constant and $\check{b}=b-\hat{b}$.
\end{thm}

\begin{proof}
Let $\hat{x}$ be the solution of 
\begin{equation*}
\ud \hat{x}_t=\hat{b}_t(\hat{x}_t)\ud t+\sigma_t(\hat{x}_t)\ud w_t,\quad \hat{x}_0=u
\end{equation*}
Then from Theorem \ref{girsdh}, for any bounded continuous function $f$ and $g$
\begin{align*}
&\e[f(x)g(L_1x_{T_1},\dots,L_Nx_{T_N})]=\e[f(\hat{x}_t)g(L_1\hat{x}_{T_1},\dots,L_N\hat{x}_{T_N})e^{\int_0^T \check{b}^*_t(\hat{x}_t)a_t(\hat{x}_t)^{-1}\ud \hat{x}_t-\frac 1 2  \|\sigma_t(\hat{x}_t)^{-1}\check{b}_t(y_t) \|^2\ud t}]\\
&=\int \e[f(\hat{x}_t)e^{\int_0^T \check{b}^*_t(\hat{x}_t)a_t(\hat{x}_t)^{-1}\ud \hat{x}_t-\frac 1 2  \|\sigma_t(\hat{x}_t)^{-1}\check{b}_t(y_t) \|^2\ud t}|(L_k\hat{x}_{t_k}=v_k)_{1\leq k\leq N}]g(v_1,\dots,v_N)\prod_k \ud v_k
\end{align*}
It remains to apply Theorem \ref{th}.
\end{proof}

\section*{Appendix}

\begin{lem}\label{monobridconv}
Let us consider Equation~(\ref{genpb}) with random initial condition $u$ on $[0,T]$ with $N=1$ which means only one observation time in $T$. 
\begin{equation*}
\ud y_t=b_t(y_t)\ud t+\sigma_t(y_t)\tilde{w}_t-P_t(y_t)\frac{y_t-u_1}{T-t}\mathbf{1}_{(T-\varepsilon_1,T)}\ud t,\quad y_0=u
\end{equation*}
Then this equation admits a unique solution on $[0,T)$. Moreover  $\|L(y_t-u_1)\|^2\leq C(\omega)(T-t)\log\log[(T-t)^{-1}+e]$ a.s., where $C$ is a positive random variable.
\end{lem}
\begin{proof}
We recall that parameters $b$ and $\sigma$ are locally Lipschitz functions. So that the equation admits a unique solution on both intervals $[0,T-\varepsilon_1]$ and $(T-\varepsilon_1,T)$ and so on $[0,T)$.
Moreover thanks to Itô's formula, on $(T-\varepsilon_1,T)$
\begin{equation*}
\ud \frac{L(y_t-u_1)}{T-t}=(T-t)^{-1}L[b_t\ud t+\sigma_t\ud \tilde{w}_t-P_t\frac{y_t-u_1}{T-t}\ud t]+L\frac{y_t-u_1}{(T-t)^2}\ud t
\end{equation*}then using (\ref{propproj}), we have $LP_t=L$ so that 
\begin{equation*}
\ud \frac{L(y_t-u_1)}{T-t}=(T-t)^{-1}L[b_t\ud t+\sigma_t\ud \tilde{w}_t]
\end{equation*}
For all $1\leq i\leq n$ the process $\{(\int_0^t(T-s)^{-1}\sigma_s(y_s)\ud \tilde{w}_s)_i,t\geq 0\}$ is a continuous local martingale whose quadratic variation $\tau_t=\int_0^t\sum_{j=1}^n(T-s)^{-2}\sigma_s(y_s)_{i,j}\ud s$ satisfies $\lim_{t\to T}\tau_t=+\infty$ and $\tau_t\leq \frac{c}{T-t}$ where $c$ is a positive constant. Hence we just have to apply the  Dambis-Dubins-Schwarz theorem that gives us the existence of a Brownian motion $B^i$ such that
\begin{equation*}
\left(\int_0^t(T-s)^{-1}\sigma_s(y_s)\ud\tilde{w}_s\right)_i=B^i(\tau_t)
\end{equation*}
The law of iterated logarithm allows us to conclude.
\end{proof}

\begin{lem}\label{monoctrls}
Let us consider Equation~(\ref{genpb}) with random initial condition $u$ on $[0,T]$ with $N=1$ which means only one observation time in $T$ 
\begin{equation*}
\ud y_t=b_t(y_t)\ud t-P_t(y_t)\frac{y_t-u_1}{T-t}\ud t,\quad y_0=u
\end{equation*}
Then for all $s<t<T$, 
\begin{equation}\label{monoctrl1}
\frac{\e[\|L(y_t-u_1)\|^2]}{T-t}\leq c(1+\sqrt{T-t}\,\e[\|L(u-u_1]\|^2])
\end{equation}
and
\begin{equation}\label{monoctrl2}
\e[\|y_s-y_t\|^2]\leq C(t-s)(1+\sqrt{T-s}\,\e[\|L(u-u_1)\|^2])
\end{equation}where $c$ and $C$ are positive constants depending on $T$, $\varepsilon_1$, bounds for $b$ and $\sigma$.
\end{lem}
\begin{proof}
Thanks to Identity (\ref{propproj}), on $(T-\varepsilon_1,T)$
\begin{equation*}
\ud L(y_t-u_1)=L[b_t\ud t+\sigma_t\ud \tilde{w}_t]-L\frac{y_t-u_1}{T-t}\ud t
\end{equation*}
Thus
\begin{equation*}
\ud \big(\|L(y_t-u_1)\|^2\big)=2(y_t-u_1)^*L^*L[b_t\ud t+\sigma_t\ud \tilde{w}_t]-2\frac{\|L(y_t-u_1)\|^2}{T-t}\ud t+\tr(La_tL^*)\ud t
\end{equation*}where the function $Tr$ gives the sum of all diagonal terms. Finally
\begin{equation*}
\ud\left(\frac{\|L(y_t-u_1)\|^2}{T-t}\right)=2\frac{(y_t-u_1)^*}{T-t}L^*L[b_t\ud t+\sigma_t \ud\tilde{w}_t]+\frac{\tr(La_tL^*)}{T-t}\ud t-\frac{\|L(y_t-u_1)\|^2}{(T-t)^2}\ud t
\end{equation*}
Setting $E_t=\e\left[\frac{\|L(y_t-u_1)\|^2}{T-t}\right]$, since $b$ and 
$\sigma$ are bounded, we get
\begin{equation}\label{ode}
E_t'\leq C_1\left(\frac{\sqrt{E_t}+1}{T-t}\right)-\frac{E_t}{T-t}
\end{equation}where $C_1$ is a positive constant depending on $\|b\|_\infty$ and $\|\sigma\|_\infty$. 
\begin{equation}\label{inqdf}
E_t'\leq (T-t)^{-1}\left[C_1\left(\frac{E_t}{2C_1}+\frac {C_1}{ 2} +1\right)-E_t\right]=(T-t)^{-1}(C-\frac{E_t}{2})
 \end{equation} where $C=C_1+\frac{C_1^2}{2}$. 
 Thus
\begin{equation*}
\left(\frac{E_t-2C}{\sqrt{T-t}}\right)'=\frac{E_t'}{\sqrt{T-t}}+\frac{E_t-2C}{2(T-t)^{\frac 3 2}}\leq 0
\end{equation*}thanks to (\ref{inqdf}).
Hence
\begin{equation*}
\frac{E_t-2C}{\sqrt{T-t}}\leq \frac{E_{T-\varepsilon_1}-2C}{\sqrt{\varepsilon_1}}
\end{equation*}that can be written
\begin{equation*}
E_t\leq 2C +\sqrt{\frac{T-t}{\varepsilon_1}}(E_{T-\varepsilon_1}-2C)
\end{equation*} 
Similarly for $t<T-\varepsilon_1$, Inequality \ref{ode} becomes
\begin{equation*}
E_{t}\leq  C'(E_0+1)\exp\{\frac{C't}{\varepsilon_1}\}
\end{equation*}where $C'$ is a positive constant only depending on $T$ and bounds of $b$ and $\sigma$. 
 That gives us (\ref{monoctrl1}).
 
 By definition for $s,t\in(T-\varepsilon_1,T)$ we have
 \begin{equation*}
y_s-y_t=\int_s^t b_\tau\ud \tau+\sigma_\tau \ud \tilde{w}_\tau-P_\tau\frac{y_\tau-u_1}{T-\tau}\ud \tau
\end{equation*} 
 Since $b$ and $\sigma$ are bounded functions, using Minkovski's inequality
 \begin{equation}\label{mk}
\e[\|y_s-y_t\|^2]^\frac 1 2\leq \e[\|\int_s^tb_\tau\ud \tau\|^2]^\frac 1 2+\e[\|\int_s^t\sigma_\tau\ud \tilde{w}_\tau\|^2]^\frac 1 2+\e[\|\int_s^tP_\tau\frac{y_\tau-u_1}{T-\tau}\ud \tau\|^2]^\frac 1 2
\end{equation}Thanks to Doob's inequality (see \emph{e.g.} \cite{IW} p.170) we get
\begin{equation*}
\e[\|\int_s^tb_\tau\ud \tau\|^2]+\e[\|\int_s^t\sigma_\tau\ud \tilde{w}_\tau\|^2]\leq C_2(t-s)
\end{equation*} where $C_2=C_1^2$ is the square of the constant introduced above.
In order to treat the last term in (\ref{mk}), we beforehand give a property
\begin{prop}\label{mink}
Let $f$ be a real-valued process defined on a segment $[a,b]$, then
\begin{equation*}
\e[(\int_a^bf_s\ud s)^2]^\frac 1 2 \leq \int_a^b\e[f_s^2]^\frac 1 2\ud s
\end{equation*}
\end{prop}
\begin{proof}
Indeed
\begin{equation*}
\e[(\int_a^bf_s\ud s)^2]\leq \e[(\int_a^b|f_s|\ud s)^2]=2\e[\int_a^b\int_a^s|f_u|\ud u|f_s|\ud s]= 2\int_a^b\int_a^s\e[|f_uf_s|]\ud u\,\ud s
\end{equation*}so that
\begin{equation*}
\e[(\int_a^bf_s\ud s)^2]\leq 2\int_a^b\int_a^s\e[(f_u)^2]^\frac{1}{2}\e[(f_s)^2]^\frac{1}{2}\ud u\,\ud s=(\int_a^b \e[f_s^2]^\frac{1}{2}\ud s)^2
\end{equation*}
\end{proof}Thanks to Proposition~\ref{mink},  assumptions (\ref{propproj}) on matrix $P$  and result (\ref{monoctrl1})
\begin{align*}
&\e\left[\Big\|\int_s^tP_\tau\frac{y_\tau-u_1}{T-\tau}\ud \tau\Big\|^2\right]\leq \left(\int_s^t\e\left[\frac{\|P_\tau(y_\tau-u_1)\|^2}{(T-\tau)^2}\right]^\frac 1 2 \ud s\right)^2\\&\leq c(1+\sqrt{T-s}\,\e[\|L(u-u_1)\|^2])(\int_s^t\frac{\ud \tau}{\sqrt{T-\tau}})^2=4c(T-s)(1+\sqrt{T-s}\,\e[\|L(u-u_1)\|^2])
\end{align*}
Finally using $(a+b+c)^2\leq 3 (a^2+b^2+c^2)$, we have on $(T-\varepsilon_1,T)$
\begin{equation*}
\e[\|y_s-y_t\|^2]\leq (C_2\wedge 4c)(T-s)(2+\sqrt{T-s}\,\e[\|L(u-u_1)\|^2])
\end{equation*}Using Doob's inequality for $s,t\in[0,T-\varepsilon_1]$ with $s<t$ 
\begin{equation*}
\e[\|y_{s}-y_t\|^2]\leq C_2(t-s)
\end{equation*}this gives the second result (\ref{monoctrl2}).
\end{proof}

\begin{lem}\label{dist}
Let $y$ and $z$ be two bridges, solutions of (\ref{genpb}) with $N=1$ and different initializations. Then, there exist two constants $C>0$ and $0<\alpha<1$ such that for all $t\in[0,T]$
\begin{equation}
\e[\|y_t-z_t\|^2]\leq C\e[\|y_0-z_0\|^2]^\alpha
\end{equation}
\end{lem}
\begin{proof}
Using their definition, we have on $[T-\varepsilon_1,T]$
\begin{equation*}
\ud (y_t-z_t)=[b_t(y_t)-b_t(z_t)]\ud t + [\sigma_t(y_t)-\sigma_t(z_t)]\ud \tilde{w}_t-\frac {P_t(y_t)(y_t-u_1)-P_t(z_t)(z_t-u_1)}{T-t}\ud t
 \end{equation*} 
 Thus
 \begin{multline}\label{itosq}
\ud \|y_t-z_t\|^2=2(y_t-z_t)^*\big[[b_t(y_t)-b_t(z_t)]\ud t + [\sigma_t(y_t)-\sigma_t(z_t)]\ud \tilde{w}_t\\-\frac {P_t(y_t)(y_t-u_1)-P_t(z_t)(z_t-u_1)}{T-t}\ud t\big]+\sum_{i,j}(\sigma_t(y_t)-\sigma_t(z_t))^2_{i,j}\ud t
 \end{multline} 
 In a same way, on $[0,T-\varepsilon_1]$ we obtain 
 \begin{equation}\label{itosq2}
\ud \|y_t-z_t\|^2=2(y_t-z_t)^*\big[[b_t(y_t)-b_t(z_t)]\ud t + [\sigma_t(y_t)-\sigma_t(z_t)]\ud \tilde{w}_t\big]+\sum_{i,j}(\sigma_t(y_t)-\sigma_t(z_t))^2_{i,j}\ud t
 \end{equation} 
 We denote $E_t=\e[\|y_t-z_t\|^2]$. We decompose the interval $[0,T]$ into $[0,T-h]$ and $(T-h,T]$ with a parameter $h$ that will be chosen later. On $[0,T-h]$ with respect to both precedent Inequations (\ref{itosq}) and (\ref{itosq2}) we get by using regularity of $b$ and $\sigma$ 
 \begin{equation*}
E_t'\leq C_1 (E_t+\frac{E_t}{T-t})\leq \frac{C_2E_t}{T-t}
 \end{equation*} where $C_1$ and $C_2$ are positive constants depending on $T$, $b$ and $\sigma$. 
 We use Gronwall's lemma to obtain
 \begin{equation*}
E_t\leq E_0\left(\frac T h\right)^{C_2}
 \end{equation*} 
 For the other part $(T-h,T]$, we use (\ref{itosq}), (\ref{itosq2}) and (\ref{monoctrl1}) to get
 \begin{equation*}
E_t'\leq C_3(E_t+\sqrt{\frac{E_t}{T-t}})\leq C_4 \frac{E_t+1}{\sqrt{T-t}}
 \end{equation*} where $C_3$ and $C_4$ are positive constants depending on $T$, $b$ and $\sigma$. Then
 \begin{equation*}
\log(E_t+1)-\log(E_{T-h}+1)\leq C_4(\sqrt h -\sqrt{T-t})\leq C_4\sqrt h
 \end{equation*} 
Finally, on $[0,T]$
\begin{equation*}
E_t\leq E_0\left(\frac T h\right)^{C_5}(1+e^{C_5\sqrt h})+e^{C_5\sqrt h}-1\leq E_0\left(\frac T h\right)^{C_5}(1+K)+C_5K\sqrt h
 \end{equation*} where $K$ is  a positive constant depending on $T$ and $C_5=C_2 \vee C_4$. We then choose $h=\left(\frac{2E_0T^{C_5}(K+1)}{K}\right)^{\frac{1}{C_5+\frac 1 2}}$ which minimizes the last member above. Hence
 \begin{equation*}
E_t\leq CE_0^{\frac{1}{2C_5+1}}
 \end{equation*} where $C$ is a positive constant depending on $b$, $\sigma$, $P$ and $T$.
 \end{proof}

\begin{proof}[Proof of Lemma~\ref{convl2} ]
We now consider an interval of type $[T_{k-1},T_{k})$. We introduce a process $y^k$ solution on this interval for the Bridge Equation~(\ref{genpb}) initialized at time $T_{k-1}$ by the value $y_{T_{k-1}}^\varepsilon$. A picture to visualize what is going on is given by Figure \ref{illus} in page \pageref{illus}.
\begin{figure}
\caption{Illustration of the three different dynamics considered}
\label{illus}
\begin{center}
\includegraphics[height=12cm,width=12cm]{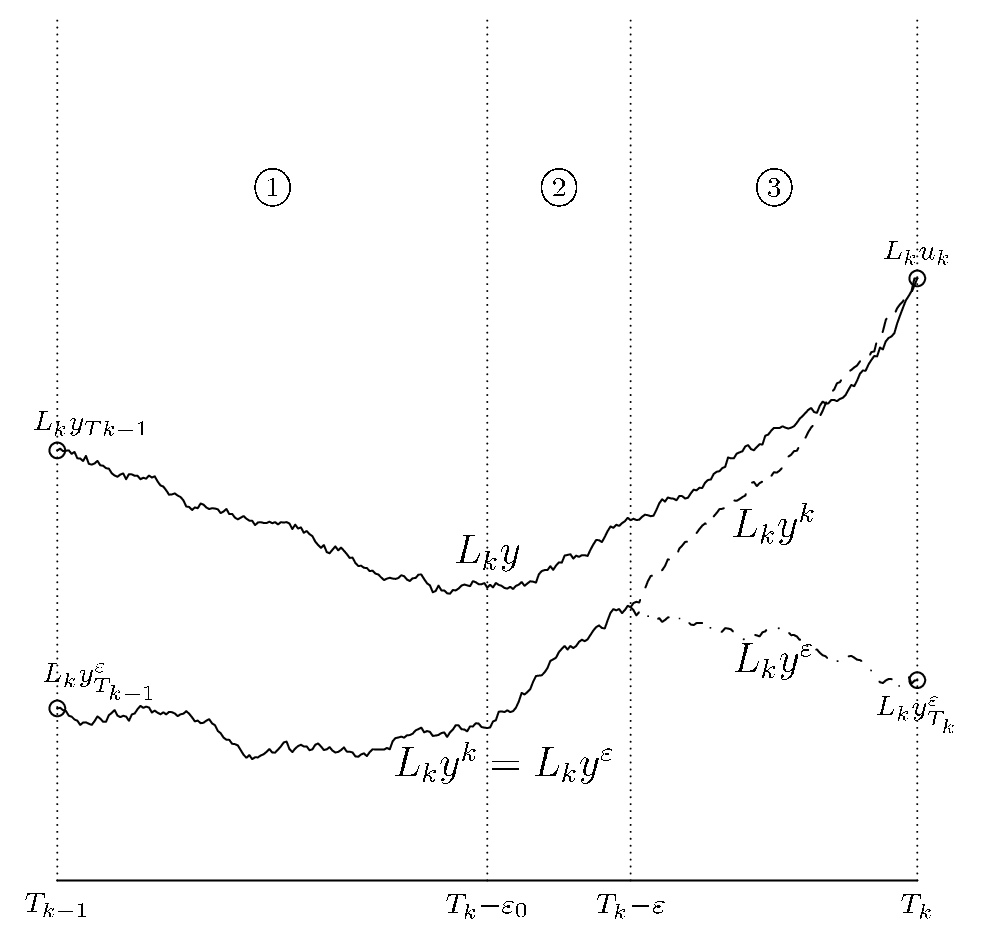}
\end{center}
Let us recall
\begin{align*}
\ud x&=b_t(x_t)\ud t+\sigma_t(x_t)\ud w_t &(\ref{init})\\
\ud y&=b_t(y_t)\ud t-P_t(y_t)\frac{y_t-u_{k}}{T_{k}-t}+\sigma_t(y_t)\ud w_t &(\ref{genpb})
\end{align*}

\begin{description}
\item[\textcircled{\scriptsize{1}}]First, the three processes follow the dynamics of the initial diffusion (\ref{init}) with a different initialization for $y$.
\item[\textcircled{\scriptsize{2}}]Now, the three processes follow the dynamics of the bridge (\ref{genpb}), that means that the correction term operates and forces these processes to get closer to the observation
\item[\textcircled{\scriptsize{3}}]At the end, only the processes $y$ and $y^k$ go on following the dynamics of the bridge (\ref{genpb}) and both tends to the obervation, while $y^\varepsilon$ follows the initial dynamics (\ref{init})
\end{description}
\end{figure}
We use this new process $y^k$ to write
\begin{equation}\label{trineq}
\e[\|y^\varepsilon_t-y_t\|^2]\leq 2\e[\|y^\varepsilon_t-y^k_t\|^2]+2\e[\|y^k_t-y_t\|^2]
\end{equation}
We will study both terms separately.

For the first one, on $[T_{k-1},T_k-\varepsilon)$ the term $\|y_t^\varepsilon-y_t^k\|$ is 0 a.s. and on $[T_k-\varepsilon,T_k)$, we can reduce the study to that of $\|x_t-y_t\|^2$ with a same initialization  $y_{T_k-\varepsilon}^\varepsilon$ at time $T_k-\varepsilon$.
\begin{equation*}
\e\big[\e[\|y^\varepsilon_t-y^k_t\|^2|y^\varepsilon_{T_{k}-\varepsilon}]\big]\leq 2\e\big[\e[\|y^k_{T_k-\varepsilon}-y^k_t\|^2|y^\varepsilon_{T_{k}-\varepsilon}]\big]+2\e\big[\e[\|y^\varepsilon_t-y^\varepsilon_{T_k-\varepsilon}\|^2|y^\varepsilon_{T_{k}-\varepsilon}]\big]
\end{equation*}
We then use Lemma (\ref{monoctrls}) given in the appendix  and classical technique (see \emph{e.g. \cite{IW}} p.170) to obtain upper bounds 
\begin{equation}\label{distxy}
\e\big[\e[\|y^\varepsilon_{t}-y^k_t\|^2|y^\varepsilon_{T_{k}-\varepsilon}]\leq c \varepsilon(1+\sqrt{\varepsilon}\,\e[\|L_k(y^\varepsilon_{T_{k-1}}-u_k)\|^2])
\end{equation}
 where $c$ is a positive constant.
Now in order to treat the remaining term we use Lemma~\ref{dist}
\begin{equation*}
\e\big[\e[\|y_t^k-y_t\|^2|y_{T_{k-1}},y_{T_{k-1}}^\varepsilon]\big]\leq \e[\|y^\varepsilon_{T_{k-1}}-y_{T_{k-1}}\|^2]^\alpha
\end{equation*}Finally, on $[T_{k-1},T_k)$
\begin{equation}\label{induction}
\e[\|y_t^\varepsilon-y_t\|^2]\leq c' \left[\varepsilon(1+\sqrt{\varepsilon}\,\e[\|L_k(y^\varepsilon_{T_{k-1}}-u_k)\|^2])+ \e[\|y^\varepsilon_{T_{k-1}}-y_{T_{k-1}}\|^2]^\alpha\right]
\end{equation}where $0<\alpha<1$ and $c'$ is a positive constant only depending on $T$, bounds $b$ and $\sigma$.
We show by induction that there exists some constant $C$ such that for all $1\leq k\leq N$
\begin{equation}\label{hyprec}
\e[ \|y_{T_k}^\varepsilon-y_{T_k} \|^2 ] \leq C^k \varepsilon
^{\alpha^{k-1}}
\end{equation}
The base case is given by Equation~(\ref{distxy}). Indeed on $[0,T_1]$ processes $y^1$ and $y$ are indistinguishable since they have a same initialization at time 0. Suppose now for some $k$ that Inequality (\ref{hyprec}) holds. We now use Equation~(\ref{induction}) to get
\begin{equation*}
\e[\|y_{T_{k+1}}^\varepsilon-y_{T_{k+1}}\|^2]\leq c' \left[\varepsilon(1+\sqrt{\varepsilon}\,\e[\|L_{k+1}(y^\varepsilon_{T_{k}}-u_{k+1})\|^2])+ \e[\|y^\varepsilon_{T_{k}}-y_{T_{k}}\|^2]^\alpha\right]
\end{equation*}Let us recall that $L_ky_{T_k}=L_ku_k$ hence
\begin{equation*}
\e[\|L_{k+1}(y^\varepsilon_{T_{k}}-u_{k+1})\|^2]\leq \e[\|L_{k+1}(y^\varepsilon_{T_{k}}-y_{T_k})\|^2]+\|L_{k+1}(u_k-u_{k+1})\|^2\leq c_k(1+\e[\|y^\varepsilon_{T_{k}}-y_{T_k}\|^2])
\end{equation*}where $c_k$ is a positive constant depending on $L_{k+1}$, $u_k$ and $u_{k+1}$. That gives us thanks to the induction hypothesis
\begin{equation*}
\e[\|y_{T_{k+1}}^\varepsilon-y_{T_{k+1}}\|^2]\leq c_k' \varepsilon(1+\sqrt{\varepsilon}\,\e[\|y^\varepsilon_{T_{k}}-y_{T_k}\|^2])+ \e[\|y^\varepsilon_{T_{k}}-y_{T_{k}}\|^2]^\alpha\leq C [\varepsilon(1+ \sqrt \varepsilon C^k \varepsilon ^{\alpha^{k-1}})+C^k \varepsilon^{\alpha^{k}}]
\end{equation*}where $C$ is a positive constant. This concludes the proof.
\end{proof}

\begin{lem}\label{condexp}
Let $(t_{k,q})_{\substack{1\leq k\leq N\\ 1\leq q\leq M_k}}$ be a sequence such that $t_{k,q}\in (T_{k-1},T_k)$ and for all  $k$, $(t_{k,q})_q$ is an increasing sequence.
Then for all bounded continuous function $g$
\begin{equation*}
\lim_{\varepsilon\to 0}\frac{\e[g(x_{t_{1,1}},\dots,x_{t_{N,M_N}})\psi^\varepsilon]}{\e[\psi^\varepsilon]}=\e[g(x_{t_{1,1}},\dots,x_{t_{N,M_N}})|(L_kx_{T_k}=v_k)_{1\leq k\leq N}]
\end{equation*}
\end{lem}

\begin{proof}
Let us recall
\begin{equation*}
C^\varepsilon\psi^\varepsilon=\prod_{k=1}^N\varepsilon^{-\frac{m_k}{2}}\eta_k^\varepsilon(x_{T_k-\varepsilon})\exp\{-\frac{\|\beta_{T_k-\varepsilon}^k(x_{T_k-\varepsilon})(L_kx_{T_k-\varepsilon}-v_k)\|^2}{2\varepsilon}\}
\end{equation*}
where for all $z\in \re ^n$
\begin{equation*}
\eta^\varepsilon_k(z)=\sqrt{\det(A^k_{T_k-\varepsilon}(z))}
\end{equation*}
Let introduce Aronson's estimates (see \emph{e.g.} \cite{Ar}, \cite{St} or \cite{DrMz}) that gives bounds for the transition density. If $p_{s,t}(u,.)$  (with $s<t$) is the density of $x_t$ knowing that $x_s=u$, we have for all $z$
\begin{equation*}
\mu (t-s)^{- \frac n 2}e^{-\frac{-\lambda \|z-u\|^2}{t-s}}<p_{s,t}(u,z)<M (t-s)^{- \frac n 2}e^{-\frac{-\Lambda \|z-u\|^2}{t-s}}
\end{equation*}
The transition densities allow to expand the density $q^\varepsilon$ of $(x_{t_{1,1}},\dots,x_{t_{N,M_N}},x_{T_1-\varepsilon},\dots,x_{T_N-\varepsilon})$ 
\begin{equation*}
q^\varepsilon(z_{1,1},\dots,z_{t_{N,M_N}},\zeta_1,\dots,\zeta_N)\\=p_{0,t_{1,1}}(u,z_{1,1})\dots p_{t_{1,M_1},T_1-\varepsilon}(z_{0,M_0},\zeta_1)\dots p_{t_{N,M_{N}},T_{N}-\varepsilon}(z_{N,M_{N}},\zeta_{N})
\end{equation*}Then we set for $\varepsilon\geq 0$
\begin{multline*}
\Phi_g^\varepsilon(\zeta_1,\dots,\zeta_N)=\e[g(x_{t_{1,1}},\dots,x_{t_{N,M_N}})|(x_{T_k-\varepsilon}=\zeta_k)_k]\\=\int g(z_{1,1},\dots,z_{t_{N,M_N}})q^\varepsilon(z_{1,1},\dots,z_{t_{N,M_N}},\zeta_1,\dots,\zeta_N)\prod _{j} \ud z_j
\end{multline*}This application is continuous according to Aronson's estimates. From this expression it comes
\begin{equation*}
\frac{I_g^\varepsilon}{I_1^\varepsilon}:=\frac{\e[g(x_{t_{1,1}},\dots,x_{t_{N,M_N}})C^\varepsilon\psi^\varepsilon]}{\e[C^\varepsilon\psi^\varepsilon]}=\frac{ C^\varepsilon\int \Phi_g^\varepsilon\prod_k\eta_k^\varepsilon\\exp \{ -\frac{\|\beta^k_{T_k-\varepsilon}(L_k\zeta_k-v_k)\|^2}{2 \varepsilon} \}\ud\zeta_k}{C^\varepsilon\int \Phi_1^\varepsilon\prod_k\eta_k^\varepsilon\exp \{-\frac{\|\beta^k_{T_k-\varepsilon}(L_k\zeta_k-v_k)\|^2\}}{2 \varepsilon}\}\ud \zeta_k}
\end{equation*}
We recall that the rows of each matrix $L_k$ form an orthonormal family. We now complete arbitrarily each family into an orthonormal basis of $\re^n$. We denote $P_k$ an arbitrary matrix whose first rows are given by $L_k$.
Then we make a basis change with respect to those matrices $P_k$ for each $\zeta_k$. Thus
\begin{equation*}
I_g^\varepsilon=C^\varepsilon\int \Phi_g^\varepsilon(P_1^{-1}\zeta_1,\dots,P_N^{-1}\zeta_N)\prod_k\eta_k^\varepsilon(P_k^{-1}\zeta_k)\exp\{-\frac{\|\beta^k_{T_k-\varepsilon}(\zeta_k^{1:m_k}-v_k)}{2 \varepsilon}\}\ud \zeta_k
\end{equation*}denoting $\zeta_k^{i:j}$ the vector composed by the coordinates from $i^\textrm{th}$ to $j^\textrm{th}$ one of $\zeta_k$. We now make a second change 
\begin{equation*}
\left\{
\begin{array}{ll}
\zeta_k^{1:m_k}=\sqrt \varepsilon \xi_k^{1:m_k}+v_k\\
\zeta_k^{m_k+1:n}=\xi_k^{m_k+1:n}
\end{array}\right.
\end{equation*}
So that
\begin{multline*}
I_g^\varepsilon=\left(\prod_k \varepsilon^{- \frac{m_k}{2}}\right) \int \Phi_g^\varepsilon(P_1^{-1}\zeta_1,\dots,P^{-1}_N\zeta
_N)\prod_k \eta_k^\varepsilon\exp\{-\frac{\|\beta^k_{T_k-\varepsilon}(\zeta_k^{1:m_k}-v_k)\|^2}{2 \varepsilon}\}\ud \zeta_k\\
=\int \Phi_g^\varepsilon\left(P_1^{-1}\left(\begin{smallmatrix}\sqrt \varepsilon \xi_1^{1:m_k}+v_1\\\xi_1^{m_k+1:n}\end{smallmatrix}\right),\dots,(P_1^{-1}\left(\begin{smallmatrix}\sqrt \varepsilon \xi_N^{1:m_k}+v_N\\\xi_N^{m_k+1:n}\end{smallmatrix}\right)\right)\prod_k\eta_k^\varepsilon  \exp\{-\frac{\|\beta_{T_k-\varepsilon}^k\xi_k^{1:m_k}\|^2}{2 }\}\ud \xi_k
\end{multline*}
We now use Aronson's estimates and Lemma~\ref{randvar} to get an integrable uniform upper bound for $q^\varepsilon$ when $0<\varepsilon<\varepsilon_0$. Thanks to Lebesgue's theorem we obtain the convergence for the last term
\begin{equation*}
I_g^\varepsilon\stackrel{\varepsilon\to 0}{\longrightarrow}\int \Phi^0_g\left(P_1^{-1}\left(\begin{smallmatrix}v_1\\\xi_1^{m_k+1:n}\end{smallmatrix}\right),\dots,(P_1^{-1}\left(\begin{smallmatrix}v_N\\\xi_N^{m_k+1:n}\end{smallmatrix}\right)\right)\prod_k\eta_k^\varepsilon  \exp\{-\frac{\|\beta_{T_k}^k\xi_k^{1:m_k}\|^2}{2 }\}\ud \xi_k 
\end{equation*}
We then integrate with respect to the $\xi_k^{1:m_k}$
\begin{equation*}
I_g^\varepsilon=\int \Phi^0_g\left(P_1^{-1}\left(\begin{smallmatrix}v_1\\\xi_1^{m_k+1:n}\end{smallmatrix}\right),\dots,(P_1^{-1}\left(\begin{smallmatrix}v_N\\\xi_N^{m_k+1:n}\end{smallmatrix}\right)\right)\prod_k\ud \xi_k^{m_k+1:n} 
\end{equation*}
Finally
\begin{equation}
\lim_{\varepsilon\to 0}\frac{I_g^0}{I_1^0}=\frac{\int \Phi^0_g\left(P_1^{-1}\left(\begin{smallmatrix}v_1\\\xi_1^{m_k+1:n}\end{smallmatrix}\right),\dots,(P_1^{-1}\left(\begin{smallmatrix}v_N\\\xi_N^{m_k+1:n}\end{smallmatrix}\right)\right)\prod_k\ud \xi_k^{m_k+1:n}}{\int \Phi^0_1\left(P_1^{-1}\left(\begin{smallmatrix}v_1\\\xi_1^{m_k+1:n}\end{smallmatrix}\right),\dots,(P_1^{-1}\left(\begin{smallmatrix}v_N\\\xi_N^{m_k+1:n}\end{smallmatrix}\right)\right)\prod_k\ud \xi_k^{m_k+1:n}}
\end{equation}We conclude thanks to the Bayes formula.
\end{proof}

\begin{lem}\label{convint}
Let $y$ be solution of Equation~(\ref{bridge}). Then almost surely for all $1\leq k\leq N$ the following integral are absolutely convergent
\begin{multline}\label{int}
\int_{T_k-\varepsilon_k}^{T_k}\frac{(L_ky_t-v_k)^*A_t^k(y_t)b_t(y_t)\ud t}{T_k-t}+\frac{(L_ky_t-v_k)^*\ud A_t^k(y_t) (L_ky_t-v_k)}{2(T_k-t)}\\+\sum_{1\leq i,j\leq m_k}\frac{\ud \big\langle A_{i,j}^k(y_.),(L_ky_.-v_k)_i(L_ky_.-v_k)_j\big\rangle_t}{2(T_k-t)}
\end{multline}
\end{lem}
\begin{proof}
We reduce the study without loss of generality to that of
\begin{equation*}
\ud y_t=b_t(y_t)\ud t+\sigma_t(y_t)\ud\tilde{w}_t-\sigma_t(y_t)\beta_t(y_t)\frac{Ly_t-v}{T-t}\mathbf{1}_{(T-\varepsilon_1,T)}(t)\ud t
\end{equation*}
We then treat integrability for each term.

For the first term, since $b$ and $\beta$ are bounded, we use Lemma~\ref{cvbridge} to get
\begin{equation*}
\left\|\frac{Ly_t-v}{T-t}\right\|\leq C\sqrt{\frac{\log\log \big((T-t)^{-1}+e\big)}{T-t}}
\end{equation*}where $C$ is a positive random variable. Now for all positive $\alpha$, we have $\log\log x\leq x^\alpha$. Then for $\alpha$ small enough, we obtain integrability of righthandside.

For the second term in (\ref{int}), we recall that for all $z$ we have $A_t(z)=\beta_t(z)^*\beta_t(z)=(La_t(z)L^*)^{-1}$, hence
\begin{equation*}
\ud A_t=p_t\ud t+q_t\ud \tilde{w}_t+r_t\frac{Ly_t-v}{T-t}\ud t
\end{equation*}where $p$, $q$ and $r$ are bounded adapted processes. So that even if it means changing $p$ $q$ and $r$
\begin{equation}\label{aexpansion}
\frac{(Ly_t-v)^*\ud A_t(Ly_t-v)}{T-t}= \frac{\|Ly_t-v\|^2}{T-t}p_t\ud t+\frac{\|Ly_t-v\|^2}{T-t}q_t\ud w_t+\frac{\|Ly_t-v\|^2}{(T-t)^2}r_t\ud t
\end{equation}Using Lemma~\ref{cvbridge}, we obtain that the quantities $\frac{\|Ly_t-v\|^2}{T-t}$, $\frac{\|Ly_t-v\|^2}{(T-t)^2}$ and $\frac{\|Ly_t-v\|^4}{(T-t)^2}$ are integrable in a left neighboorhood of $T$.

For the last term in (\ref{int}), we use Itô's formula and the fact that $L\sigma_t(z)\beta_t(z)=I_d$, so that on $(T-\varepsilon_k,T_k)$
\begin{equation*}
\ud (Ly_t-v)=L[b_t\ud t+\sigma_t\ud\tilde{w}_t]-\frac{Ly_t-v}{T-t}\ud t
\end{equation*}Hence
\begin{equation*}
\ud \big\langle A_{i,j},(Ly_.-v)_i(Ly_.-v)_j\big\rangle_t\leq \|Ly_t-v\|p_t\ud t
\end{equation*}where $p$ is the same bounded adapted process given above. Finally
\begin{equation*}
\sum_{i,j}\frac{\ud \big\langle A_{i,j},(Ly_.-v)_i(Ly_.-v)_j\big\rangle_t}{T-t}\leq \frac{\|Ly_t-v\|p_t}{T-t}\ud t
\end{equation*}even if it means changing $p$, and this last term is integrable.
\end{proof}

\begin{lem}\label{randvar}
Let $(Z_j)_{1\leq j\leq K}$ be a family of random $m_j$-dimensional variables and let $(g_j : \re^{m_j}\to \re)_{1\leq j\leq K}$ be a family of densities. Then the function
\begin{equation*}
\begin{array}{cl}
\prod_{j=1}^K\re^{m_j} &\rightarrow \re\\
(v_j)_j &\mapsto \e\left[\prod_jg_j(Z_j-v_j)\right]
\end{array}
\end{equation*}is the density of the family $(V_j=W_j+Z_j)_{1\leq j\leq K}$ where each of the $W_j$ whose law is given by $g_j$ is independent with respect to the $(Z_j)_{1\leq j\leq K}$ and $(W_k)_{k\neq j}$.
\end{lem}
\begin{proof}
Let $f$ be a bounded continuous function
\begin{equation*}
\int f\big((v_j)_j\big)\e\left[\prod_j g_j (Z_j-v_j)\right]\prod_j \ud v_j=\e\left[\int f\big((v_j)_j\big)\prod_j g_j (Z_j-v_j)\right]\prod_j \ud v_j
\end{equation*}Then we make the change of variables $w_j=Z_j-v_j$ for all $j$
\begin{equation*}
\int f\big((v_j)_j\big)\e\left[\prod_j g_j (Z_j-v_j)\right]\prod_j \ud v_j=
\e\left[\int f\big((w_j+Z_j)_j\big)\prod_j g_j (w_j)\right]\prod_j \ud w_j
\end{equation*}Hence
\begin{equation*}
\e\left[\int f\big((w_j+Z_j)_j\big)\prod_j g_j (w_j)\right]\prod_j \ud w_j=\e\left[f\big((W_j+Z_j)_j\big)\right]
\end{equation*}where $W_j$ admits $g_j$ as density.
\end{proof}
\bibliographystyle{plain}
\bibliography{biblio}
\end{document}